\documentclass[]{siamart1116}
\usepackage[placement=top,angle=90,color=black,scale=1.2]{background}
\backgroundsetup{position={0.41\textwidth,1.5cm},angle=0,contents={\fboxsep5pt\colorbox{white}{\parbox{0.833\textwidth}{\centering \textsc{PREPRINT} (2017) \\Friedrich-Alexander Universit\"at Erlangen-N\"urnberg (FAU)}}}}
\usepackage[utf8]{inputenc}
\usepackage[T1]{fontenc}
\usepackage{hyphenat}
\usepackage{booktabs,longtable}
\usepackage{amssymb}
\usepackage{dsfont}
\usepackage{mathtools}
\usepackage[version-1-compatibility]{siunitx}
\usepackage{xspace}
\usepackage{subfigure}
\usepackage{algpseudocode}
\usepackage{enumitem}

\let\oldc\c

\newcommand     {\C}    {\mathds C}
\newcommand     {\R}    {\mathds R}
\renewcommand   {\S}    {\mathbb S}
\newcommand     {\N}    {\mathds N}

\newcommand     {\eps}  {\varepsilon}

\renewcommand   {\phi}  {\varphi}

\newcommand     {\mo}           {{\text{-}1}}
\newcommand     {\rmd}          {\,\mathrm d}
\newcommand     {\dx}           {\,\mathrm dx}
\renewcommand   {\d}            {\partial} 
\newcommand     {\del}[2]       {\frac{\d #1}{\d #2}}
\newcommand     {\tot}[2]       {\frac{\mathrm d #1}{\mathrm d #2}}
\newcommand     {\secdel}[2]    {\frac{\partial^2 #1}{\partial^2 #2}}

\newcommand     \tr     {\operatorname{tr}}
\newcommand     \st     {\operatorname{s.t.}}
\newcommand     \diag   {\operatorname{diag}}
\newcommand     \PML    {\text{PML}}
\newcommand     \sign   {\operatorname{sign}}
\renewcommand   \mod    {\operatorname{mod}}

\renewcommand   \Re     {\operatorname{Re}}
\renewcommand   \Im     {\operatorname{Im}}

\newcommand     \adj    {\operatorname{adj}}

\renewcommand   \b      {\boldsymbol}
\newcommand     \B      {{\boldsymbol B}}
\renewcommand   \u      {{\boldsymbol u}}
\newcommand     \Y      {\boldsymbol Y}
\newcommand     \D      {\boldsymbol D}
\newcommand     \E      {\boldsymbol E}

\newcommand\rmG{\mathrm G}

\newcommand{\Jphys}                 {\ensuremath{J^p}\xspace}
\newcommand{\Jreg}                  {\ensuremath{J^r}\xspace}
\newcommand{\Jgray}                 {\ensuremath{J^g}\xspace}
\newcommand{\Jphysdisc}             {\ensuremath{J^p_h}\xspace}
\newcommand{\Jregdisc}              {\ensuremath{J^r_h}\xspace}
\newcommand{\Jgraydisc}             {\ensuremath{J^g_h}\xspace}
\newcommand{\nonsep}                {\ensuremath{f}\xspace}
\newcommand{\hyperApproxArg}[2]     {\ensuremath{\nonsep^{#1}_{#2}}\xspace}
\newcommand{\partHyperApprox}[4]    {\ensuremath{\nonsep^{#1,#2}_{#3,#4}}\xspace}
\newcommand{\constHyperApprox}[3]   {\ensuremath{c^{#1}_{#2,#3}}\xspace}
\newcommand{\hyperApprox}           {\ensuremath{\hyperApproxArg{\tau}{\bar\B}}\xspace}
\newcommand{\partHyperApproxIR}     {\ensuremath{\partHyperApprox{i}{\tau}{\bar\B}{R}}\xspace}
\newcommand{\partHyperApproxII}     {\ensuremath{\partHyperApprox{i}{\tau}{\bar\B}{I}}\xspace}
\newcommand{\partHyperApproxIs}     {\ensuremath{\partHyperApprox{i}{\tau}{\bar\B}{s}}\xspace}
\newcommand{\constHyperApproxIR}    {\ensuremath{\constHyperApprox{i}{\bar\B}{R}}\xspace}
\newcommand{\constHyperApproxII}    {\ensuremath{\constHyperApprox{i}{\bar\B}{I}}\xspace}
\newcommand{\constHyperApproxIs}    {\ensuremath{\constHyperApprox{i}{\bar\B}{s}}\xspace}
\newcommand{\Jtotal}                {\ensuremath{\mathcal J}\xspace}
\newcommand{\totalHyperApprox}      {\ensuremath{\mathcal J^\tau_{\bar \B}}\xspace}
\newcommand{\totalHyperApproxArg}[2]{\ensuremath{\mathcal J^{#1}_{#2}}\xspace}
\newcommand{\constParamApproxIRU}   {\ensuremath{C^{i,R}_{U}}\xspace}

\newcommand{\constParamApproxIsU}   {\ensuremath{C^{i,s}_{U}}\xspace}
\newcommand{\constParamApproxIRL}   {\ensuremath{C^{i,R}_{L}}\xspace}

\newcommand{\constParamApproxIsL}   {\ensuremath{C^{i,s}_{L}}\xspace}
\newcommand{\noElements}            {\ensuremath{{N_T}}\xspace}
\newcommand{\noDesign}              {\ensuremath{{K}}\xspace}
\newcommand{\noEdges}               {\ensuremath{{N_E}}\xspace}

\newcommand{\noDoF}                 {\ensuremath{{N_p}}\xspace}
\newcommand{\noDir}                 {\ensuremath{{N_d}}\xspace}
\newcommand{\noWaveL}               {\ensuremath{{N_\omega}}\xspace}

\newcommand{\indexEdges}            {\ensuremath{\mathcal I_E}\xspace}
\newcommand{\indexDesign}           {\ensuremath{\mathcal I_D}\xspace}
\newcommand{\Param}                 {\ensuremath{\psi}\xspace}
\newcommand{\Triang}                {\ensuremath{\boldsymbol T}\xspace}
\newcommand{\TriI}[1]               {\ensuremath{T_{#1}}\xspace}
\newcommand{\RHS}                   {\ensuremath{H}\xspace}
\newcommand{\SYS}                   {\ensuremath{S}\xspace}

\newcommand{\TangCone}              {\ensuremath{\mathcal T}\xspace}
\newcommand{\GraphSet}              {\ensuremath{\mathcal G}\xspace}
\newcommand{\AdmSet}                {\ensuremath{\mathcal B}\xspace}
\newcommand{\EdgeImage}             {\ensuremath{\mathcal E}\xspace}
\newcommand{\NodeImage}             {\ensuremath{\mathcal N}\xspace}
\newcommand{\SC}                    {\ensuremath{\S_{\C}}\xspace}
\newcommand{\SCK}                   {\ensuremath{\S^{\noDesign}_{\C}}\xspace}
\newcommand{\SR}                    {\ensuremath{\S_{\R}}\xspace}

\newcommand{\ie}                    {{i.\,e.}\xspace}
\newcommand{\FOOptConst}            {\ensuremath{\nu}}

\newtheorem{remark}[theorem]{Remark}
\newtheorem{ass}[theorem]{Assumption}

\title{Material Optimization in Transverse Electromagnetic Scattering Applications}
\author{%
    Johannes Semmler%
    \thanks{Applied Mathematics 2, Friedrich-Alexander University Erlangen-N\"urnberg (FAU), Germany (\email{johannes.semmler@fau.de})}%
    \and Lukas Pflug%
    \and Michael Stingl}
\headers{Material Optimization in Electromagnetics}{J. Semmler, L. Pflug and M. Stingl}

 \algrenewtext{EndLoop}{\algorithmicend}
    \algrenewtext{EndIf}{\algorithmicend}
    \algrenewcommand\algorithmicthen{}

\newcommand\corr[1]{#1}
\begin{document}\maketitle\begin{abstract} A class of algorithms for the solution of discrete material optimization problems in electromagnetic applications is discussed. \corr{The idea behind the algorithm is similar to that of the sequential programming. However,} in each major iteration a model is established on the basis of an appropriately parametrized material tensor. The resulting nonlinear parametrization is treated on the level of the sub-problem, for which, \corr{globally optimal solutions can be computed due to the block separability of the model}. Although global optimization of non-convex design problems is \corr{generally} prohibitive, a \corr{well chosen} combination of analytic solutions along with standard global optimization techniques leads to a very efficient algorithm \corr{for most} relevant material parametrizations. A global convergence result for the overall algorithm is established. The effectiveness of the approach in terms of \corr{both computation time and solution quality} is demonstrated by numerical examples\corr{,} including the optimal design of cloaking layers for a nano-particle \corr{and} the identification of multiple materials with different optical properties in a matrix. \end{abstract} \section{Introduction} Problems of material optimization governed by Maxwell's equation \corr{have recently been} studied in the literature. In particular, for time-harmonic electromagnetic fields we refer to \cite{Diaz2010}, where an optimal distribution of two materials with distinct properties was computed based on the \corr{so-called} SIMP approach \cite{Bendsoe1994}\corr{. This approach} was originally developed for the topology optimization of elastic structures and is based on interpolation between the desired material properties and an appropriate penalization scheme rendering undesired intermediate material properties unattractive with respect to the particular cost function. A similar technique has been applied to the transient problem discussed, \corr{for example} in \cite{Hassan2014}. Again the goal here was to find an optimal distribution of two isotropic materials. Potential applications of structural optimization techniques in the context of electromagnetics range from inverse problems, where distribution of material is reconstructed by the information given by the scattered electromagnetic fields \cite{OhinKwon2002}, to optimal material layout to improve the properties of optical devices \cite{Byun2004} or nanoparticles \cite{Pendry}. \par In this paper we are interested in a more general class of material optimization problems, in the framework of which a complex-valued permittivity tensor for a given point in the design domain is specified by a function of a finite number of parameters. Particular \corr{realizations have led} to problems of free material optimization \cite{Zowe1997,Greifenstein2016,Bendsoe1994,Ringertz1993}, which \corr{have so far} been studied solely in the context of linear elasticity, to optimal material orientation problems, \corr{see} e.g. \cite{Pedersen1989}, and to so called discrete multi-material optimization as treated in literature by so called DMO methods, see, e.g. \cite{Stegmann2005,Hvejsel2011}. Rather than formulating the optimization problems \corr{directly in the design parameters} and using a derivative based optimization algorithm like SNOPT \cite{Gill2005a} or MMA \cite{Svanberg1987} in a ``black box'' way, in this article, a new algorithmic concept for the solution of the envisaged class of design problems is developed. The motivation for the development of \corr{this} new solution approach is \corr{the fact that} the material tensors typically depend on the design parameters in a non-linear way and thus the parametrization may \corr{result in numerous} poor local optima, see \cite{Pedersen1989}, in which algorithms applied in a black-box way may \corr{become} trapped. \par In order to prevent this, the following concept is suggested: the \corr{principal} idea is to formulate the design problem \corr{directly in terms of the material tensors}, while the associated parametrization is hidden in the definition of the admissible set. Then, in the course of a sequential approximation algorithm, FMO-type models \corr{(see \cite{Stingl2009a})} of the objective \corr{as a} function of the material tensors are derived and \corr{are used to generate a sequence of sub-problems.} Due to the potentially non-convex parametrization each sub-problem is a constrained non-linear optimization problem, which may exhibit an unknown number of local optima. We show that based on the properties of the particular approximations these sub-problems can be solved to global optimality with a reasonable effort, partially with \corr{the} analytical solution, for important classes of parametrizations. \par The manuscript is structured as follows: In \cref{sec:prereq} the Helmholtz-type state equation, based on the time-harmonic Maxwells equation, is given in its weak formulation and the dependency on the material tensor is \corr{highlighted}. Then, in \cref{sec:optimization} the class of optimization problems of interest is stated\corr{,} including a detailed description of the general structure of the objective function as well as \corr{the} structure of the set of admissible materials, which is based on a graph. The discussion is continued \corr{with} a short note on the discretization of the state and the optimization problem as well as regularization issues. \corr{\Cref{sec:algorithm}} \corr{constitutes} the heart of this article. Based on convex first-order hyperbolic approximation as well as a \corr{so-called} sequential global programming technique, an optimization algorithm is \corr{stated for} which a global convergence result can be established. Subsequently, parametrization-dependent solutions \corr{to the} sub-problems taking the graph-structure of the admissible set into account are derived. To show the capabilities of the algorithm\corr{,} two examples are discussed in \cref{sec:example}. These \corr{include both} the design of a cloaking for a scatterer made from an increasing number of anisotropic materials \corr{and the} tomographic reconstruction of an unknown material distribution consisting of a background material, a dielectric and an absorbing material. \par Throughout this paper, we indicate by $\SC$ the space of symmetric, two\hyp{}dimensional and complex-valued tensors. The term \corr{$\langle A,B \rangle := \Re(\tr(A^H B))$ denotes the standard scalar product in $\SC$ and $\| A \|^2_F = \langle A,A\rangle$ denotes the induced Frobenius norm}. \par For a real-valued continuously differentiable function $v:\SC \to \R$ we define the derivative of $v$ in a direction $Y\in \SC$ with respect to $B$ as \begin{align*} \del{v(B)}{B}[Y] := \lim_{\nu \to 0} \frac{v(B+\tau Y) - v(B)}{\nu}. \end{align*} We note that for the tuple $\B\in \SCK$ with $\noDesign\in \N$ and a real-valued continuously differentiable function $v:\SCK \to \R$\corr{, the} directional derivative of $v$ in direction $\Y \in \SCK$ with respect to $\B$ is given as \begin{align*} \del{v(\B)}{\B}[\Y] = \sum_{i=1}^\noDesign \del{v(\B)}{(\B)_i}[(\Y)_i]. \end{align*} Finally, we define the extended norm $\|\B\|^2_{F^\noDesign} := \sum_{i=1}^\noDesign \|(\B)_i\|^2_F$ for $\B\in \SCK$.\section{Prerequisites}\label{sec:prereq} The propagation of electromagnetic waves is described by Maxwell's equation \cite{Jackson1963}. In this paper we restrict ourselves to the time-harmonic propagation of so-called transverse magnetic waves (TM) for a given wavenumber $\omega$, where we assume that the electromagnetic field is given by a scalar function depending only on two spatial dimensions. With these assumptions, Maxwell's equation simplifies to the Helmholtz equation for the magnetic field. \par The relative permittivity $\eps$, which in this article is the material property of interest, is a complex- and tensor-valued function of space. For modeling purposes an additional tensor valued function $B$ is introduced\corr{,} whose values are given by the inverse of the permittivity \corr{at} each point. In general, we assume that the material tensor at a point is symmetric, \ie $B: \R^2 \to \SC$. For scattering applications an incident magnetic field $u_I:\R^2\to \C$ is given, which solves Maxwell's equation for the given background material $B_b: \R^2 \to \SC$. \corr{The} Helmholtz equation is actually defined \corr{on the whole of} $\R^2$, thus we introduce a perfectly matched layer (PML) \cite{Berenger1994} surrounding the domain of interest\corr{,} including the scattering object. \par \begin{figure}[t]\centering \includegraphics{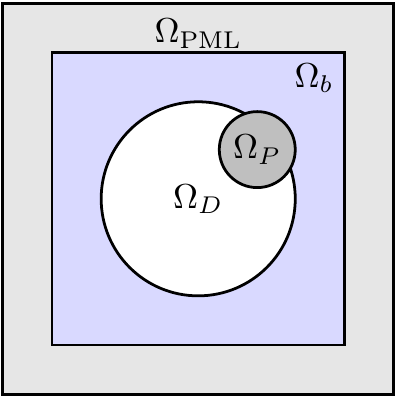} \caption{General geometrical setting: PML, background and design domain} \label{fig:domains} \end{figure} \par The computational domain $\Omega = \Omega_D \cup \Omega_C\subset\R^2$ is subdivided into a design domain $\Omega_D$ and a non-design domain $\Omega_C$. \corr{The} non-design domain $\Omega_C = \Omega_\PML\cup \Omega_b \cup \Omega_P$ \corr{in turn} consists of three subdomains. The perfectly matched layer $\Omega_\PML$ \corr{completely} encloses the background domain $\Omega_b$ and both are equipped with background material tensor-valued function $B_b$. Moreover\corr{,} the scatterer domain $\Omega_P$ and the design domain $\Omega_D$ are both embedded in the background domain $\Omega_b$ (see \cref{fig:domains}). The tensor valued function $B_P:\Omega_P\to\SC$ in the scatterer domain is assumed to be independent of the design, whereas the tensor function $B:\Omega_D\to \SC$ associated with the design domain $\Omega_D$ will be subject to optimization. For the sake of notation, we combine both functions to the piecewise tensor-valued function $B_C:\Omega_C\to \SC$ with \begin{align*} B_C & = \begin{cases} B_b & \text{ in } \Omega_b \cup \Omega_\PML, \\ B_P & \text{ in } \Omega_P. \end{cases} \end{align*} Using this\corr{,} we state the Helmholtz equation in weak form: \begin{align}\label{eqn:weakbilinear} {\text{Find } u \in H^1_0(\Omega,\C) \ \st } & a(B;u,\phi) + a_C(u,\phi) = l(B;\phi)+l_C(\phi) \quad \text{ for all } \phi\in H^1_0(\Omega)\corr{.} \nonumber \end{align} Here, we explicitly point out the dependency on $B$ and note the subdivision of the bilinear and linear forms into design domain contributions \begin{align*} a(B;u,\phi) & = \int_{\Omega_D} B \nabla u \cdot \nabla \phi - \omega^2 u\phi \dx, & l(B;\phi) & = -\int_{\Omega_D} B\nabla u_I\cdot \nabla \phi \dx \end{align*} and non-design domain contributions \begin{align*} a_C(u,\phi) & = \int_{\Omega_C} B_C A^\omega_\eps \nabla u \cdot \nabla \phi - \omega^2 A^\omega_\mu u\phi \dx, \\ l_C(\phi) & = \int_{\Omega} B_b\nabla u_I\cdot \nabla \phi \dx-\int_{\Omega_C} B_C \nabla u_I\cdot \nabla \phi \dx. \end{align*} By the definition of $B_C$ we observe that the right hand side of \eqref{eqn:weakbilinear} \corr{vanishes} in $\Omega_b\cup \Omega_\PML$. The functions $A^\omega_\eps$ and $A^\omega_\mu$ describe the wavelength dependent PML \cite{Monk2003} based on a squared layer with distance $d$ from the origin and are defined as follows: \begin{align*} A^\omega_\epsilon(x,y) &= \begin{pmatrix} \frac{s(y)}{s(x)} & 0 & \\ 0 & \frac{s(x)}{s(y)} \end{pmatrix}, & A^\omega_\mu(x,y) & = s(x) s(y), & s(t) &= 1 - \frac{ \sigma_0 \max(0,|t|-d)}{\imath\omega}. \end{align*} Particular choices of the positive scalar $\sigma_0$ and $d$ depend on the particular application, see \cref{sec:example}. The definition of $s$ implies that $A^\omega_\epsilon\equiv\mathds1$ and $A^\omega_\mu \equiv 1$ in $\Omega \setminus \Omega_\PML$. \section{A general material optimization problem}\label{sec:optimization} \par \begin{figure}\centering \scalebox{0.7}{ \includegraphics[]{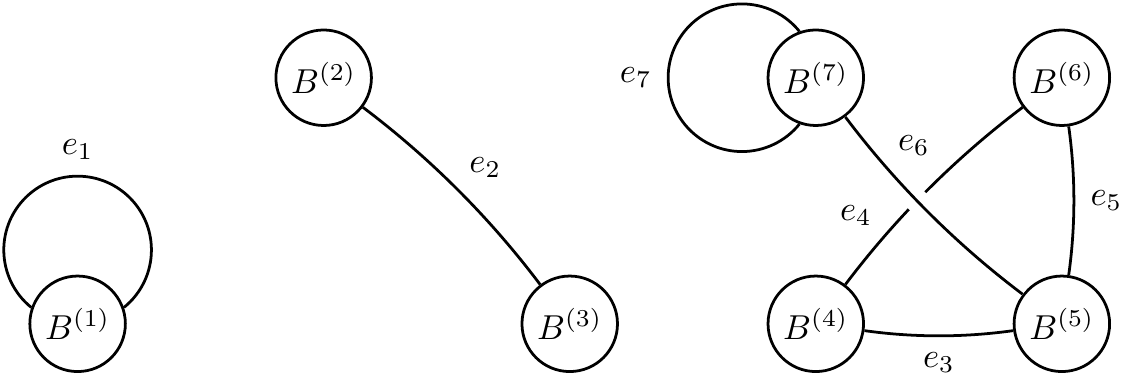} } \caption{Example graph of admissible set $\GraphSet$}\label{fig:graph} \end{figure} \par We start with the description of the set of admissible material tensors $\GraphSet$, which is structured by a graph $(V,E)$ with vertices $V$ and edges $E = \{e_1,\dotsc,e_\noEdges\} \subset V\times V$. We assume that every vertex $v\in V$ is part of at least one edge and is associated with a predefined material tensor $B^{(v)}\in \SC$, cf. \cref{fig:graph}. We further define $\NodeImage:=\{B^{(v)}\mid v\in V\}$, $\indexEdges := \{1,\dotsc, \noEdges\}$ and introduce the following parametrization: \begin{definition}[parametrization of $\GraphSet$]\label{ass:properties} We call the mapping $\Param: \indexEdges \times [0,1] \to \SC$ parametrization of $\GraphSet$ iff the following holds: \begin{itemize} \item $\Param$ is twice continuously differentiable with respect to the second variable. \item interpolation property: \begin{equation*} \Param(l,0) = B^{e_l^{(1)}} \text{ and } \Param(l,1) = B^{e_l^{(2)}} \quad \forall l\in\indexEdges, \end{equation*} where $e_l^{(1)}$ and $e_l^{(2)}$ denote the first and second node of the $l$-th edge, respectively. \item $\Param$ is injective on $\indexEdges\times(0,1)$, \ie \begin{equation*} \forall (k,s),(l,t)\in \indexEdges\times(0,1): \Param(k,s) = \Param(l,t) \Rightarrow (k,s)=(l,t). \end{equation*} \end{itemize} \end{definition} We denote \corr{the image generated by the parametrization on the $l$-th edge by $\EdgeImage_l$}, \ie \begin{align*} \EdgeImage_l = \{\Param(l,\delta)\mid \delta \in [0,1]\}, \quad \forall l\in\indexEdges. \end{align*} Thus the set of admissible material tensors can be written as \begin{align*} \GraphSet = \{ \Param(l,\delta) \mid (l,\delta) \in \indexEdges\times [0,1] \} = \EdgeImage_1 \cup \dotsb \cup \EdgeImage_\noEdges \end{align*} and the set $\rmG$ of admissible tensor-valued material distribution reads as \begin{equation*} \rmG = \{ B :\Omega_D \to \SC \mid B(x) \in \GraphSet \text{ for } x \in \Omega_D\}. \end{equation*} Particular choices of $\rmG$ are given in \cref{sec:example}. In the following, we use the notation $\Param_l(\delta):=\Param(l,\delta)$ for all $l\in \indexEdges, \delta\in [0,1]$. Using this, we state an optimal design problem of tensor-valued coefficients over a Helm\-holtz-type equation in the domain $\Omega$ as follows: \begin{equation} \label{eqn:opt:prob}\left\{\quad \begin{aligned} \min_{B\in \rmG} ~ & \Jphys(B, u) + \eta \Jreg(B) + \gamma \Jgray(B) \\ \st~& u\in H^1_0(\Omega,\C) \text{ is the solution of } & \\ & \begin{multlined} a(B;u,\phi) + a_C(u,\phi) = l(B;\phi)+l_C(\phi) \quad \text{ for all } \phi\in H^1_0(\Omega) \end{multlined} \end{aligned}\right. \end{equation} The real-valued function $\Jphys$ is called \corr{the} objective functional and \corr{is} assumed to be Gateaux-differentiable on $\rmG \times H^1_0(\Omega,\C)$ with respect to \corr{both the} complex- and tensor-valued material distribution $B$ and the state variable $u$. \par Moreover\corr{,} $\Jreg$ and $\Jgray$ are Gateaux-diffentiable functionals \corr{that} penalize irregular and undesired material distributions $B$, respectively, and $\eta$ and $\gamma$ are non-negative scalars. We refer to \cref{sec:reg,sec:subprobsolve} for precise definitions. \subsection{Discretization}\label{sec:discrete} Let $\Triang$ be a regular triangulation of $\Omega$ with $\noElements$ triangular elements $\TriI{i}$, $1\le i\le \noElements$\corr{, where} the first $\noDesign$ triangles are located in the design domain $\Omega_D$. We assume that all subdomains in $\Omega$ are exactly approximated by the triangulation. \par The tensor-valued function $B$ is assumed to be constant on each triangle $\TriI{i}$, $i\in \indexDesign:=\{1,\dotsc,\noDesign\}$ of the design domain, \ie $B(x)| _{\Omega_D} = B_h(x):= \sum_{i=1}^\noDesign (\B)_i \chi_{\TriI{i}}(x)$ with a tuple $\B\in \SCK $ and the characteristic function $\chi_{\TriI{i}}$ of triangle $\TriI{i}$. Furthermore\corr{,} the state $u:\Omega\to \C$ is approximated by $u_h(x):=\sum_{i=1}^\noDoF (\u)_i \phi_i(x)$ with $\noDoF$ degrees of freedom entering the coefficient vector $\u \in \C^\noDoF$. In summary\corr{,} the optimization problem \eqref{eqn:opt:prob} is approximated by \begin{align}\label{eqn:opt:probTMdis} \left . \begin{aligned} \min_{\B \in \AdmSet} ~ & \Jphysdisc(\B, \u ) + \eta \Jregdisc(\B) + \gamma \Jgraydisc(\B) \\ \st ~ & (\SYS(\B) + \SYS_C ) \u = (\RHS(\B) + \RHS_C) \end{aligned}\qquad \right. \end{align} with the admissible set $\AdmSet := (\GraphSet)^\noDesign$ and discretized versions of the objective functional $\Jphysdisc(\B,\u) := \Jphys(B_h,u_h)$, regularization functional $\Jregdisc(\B):=\Jreg(B_h)$ and penalization functional $\Jgraydisc(\B):=\Jgray(B_h)$. The PDE constraint is approximated by a system of equations with the symmetric matrices $\SYS(\B),\SYS_C\in \C^{\noDoF\times\noDoF}$, which are defined entry-wise by \begin{align*} (\SYS(\B))_{ij} & = \sum_{k=1}^\noDesign\int_{\TriI k} B_k \nabla \phi_i \cdot \nabla \phi_j -\omega^2 \phi_i\phi_j \rmd x , & \text{for all } 1\le i,j \le \noDoF \\ (\SYS_C)_{ij} & = \int_{\Omega_C} B_C A^\omega_\epsilon \nabla \phi_i \cdot \nabla \phi_j - \omega^2 A^\omega_\mu \phi_i\phi_j \rmd x & \text{for all }1\le i,j \le \noDoF. \end{align*} Finally, the right hand side vectors $\RHS(\B),\RHS_C \in \C^\noDoF$ are defined by \begin{align*} (\RHS(\B))_{i} & = \sum_{k=1}^\noDesign\int_{\TriI k} B_k \nabla u_I\cdot \nabla \phi_i \rmd x & \text{for all } 1\le i \le \noDoF \\ (\RHS_C)_{i} & = \int_{\Omega_C} B_C\nabla u_I\cdot \nabla \phi_i \rmd x - \int_{\Omega} B_b \nabla u_I\cdot \nabla \phi_i \rmd x & \text{for all } 1\le i,j \le \noDoF. \end{align*} We note that the solution $\u$ of the discretized state problem \begin{align}\label{eqn:state:algebraic} (\SYS(\B) + \SYS_C ) \u = \RHS(\B) + \RHS_C \end{align} is uniquely defined by the material tuple $\B \in \AdmSet$ and thus we can rewrite the discretized optimization problem \eqref{eqn:opt:probTMdis} as \begin{align}\label{eqn:opt:probTMdisred} \min_{\B\in\AdmSet}~\Jtotal(\B) \tag{\text{$P^{\gamma}_h$}} \end{align} with the objective functional $\Jtotal(\B) := \Jphysdisc(\B, \u(\B)) + \eta \Jregdisc(\B) + \gamma \Jgraydisc(\B)$, where $\u(\B)$ is the unique solution of \cref{{eqn:state:algebraic}}. For later use\corr{,} we define $\nonsep(\B):= \Jphysdisc(\B, \u(\B)) + \eta \Jregdisc(\B)$ and \corr{briefly} note that the derivative of $\Jphysdisc(\B,\b u( \B))$ with respect to $\B$ can be computed by adjoint calculus or \corr{using the} implicit function theorem.\subsection{Regularizations}\label{sec:reg} In this section we present a possible regularization on the material distribution $B:\Omega_D\to \SC$. \corr{To do this,} we define the filtered material distribution \begin{align*} \tilde B(x) := \frac{\int_{\Omega_D} \kappa(x-y)B(y) \rmd y}{\int_{\Omega_D} \kappa(x-y)\rmd y} \end{align*} with a filter kernel $\kappa:\R^2\to\R$. For instance $\kappa$ can be defined as a "circular filter" $x\mapsto \max(0,r_0-\|x\| )$ with filter radius $r_0$. \corr{In this way,} we define a tracking-type filter regularization term on the design domain $\Omega_D$ \begin{align*} \Jreg(B) = \int_{\Omega_D} \big \| B(x) - \tilde B (x) \big \|_F^2 \dx. \end{align*} After discretization, the regularization term can be \corr{expressed in} quadratic form \begin{equation*} \Jregdisc(\B) = \sum_{i,j=1}^2 \B_{ij}^H M \B_{ij} \end{equation*} where $\B_{ij} := (((\B)_1)_{ij},\dotsc,((\B)_\noDesign)_{ij})$. The particular form of the symmetric matrix $M\in \R^{\noDesign\times \noDesign}$ depends on $\kappa$ and on the chosen quadrature rule. Again for later use, the directional derivative with $\Y\in \SCK$ and $\B\in\SCK$ of the discretized regularization term reads for all $k\in\indexDesign$ \begin{equation*} \del{\Jregdisc(\B)}{(\B)_k}[(\Y)_k] = 2 \Re \left (\sum_{i,j=1}^2 ((\Y)_k)_{ij} \sum_{l=1}^\noDesign M_{kl} ((\B)_l)_{ij} \right). \end{equation*} It is well known that for the existence of a solution of \eqref{eqn:opt:prob} in infinite dimensions, \corr{an $L^2$-regularization is generally not,} sufficient. Nevertheless\corr{,} we see the desired regularization effect in the discretized setting, compare also \corr{with} \cref{sec:example}. We finally note that we do not apply a standard filter scheme as in \cite{Haber1996a,Sigmund1997a}, because the filtered material distribution $\tilde B$ is typically not contained in the admissible set $\GraphSet$. \section{Optimization Algorithm}\label{sec:algorithm} Throughout this section we would like to derive an optimization algorithm which takes the specific structure of problem \cref{eqn:opt:probTMdis} into \corr{consideration}. We proceed as follows: we first define suitable separable approximations of the non-separable functions $\nonsep(\B)$ defined in the previous section. On the basis of these we define a series of sub-problems and the so-called sequential global programming algorithm. \corr{We then} derive a global convergence result for the latter and show how we can efficiently solve non-convex separable sub-problems for two specific material parametrization schemes. \par We \corr{begin with some additional} notation: let $v:\SCK \to \R$ be continuously differentiable on a subset $\AdmSet\subset \SCK$. For all $i \in \indexDesign$ we define real and imaginary differential operators entry-wise by \begin{align*} (\nabla ^{i,R} v(\B))_{k,l} & := \del{v(\B)}{(\B)_i} [\boldsymbol e_k \boldsymbol e_l^T], & (\nabla ^{i,I} v(\B))_{k,l} & := \del{v(\B)}{(\B)_i} [\imath (\boldsymbol e_k \boldsymbol e_l^T)] \quad \forall 1 \le k,l\le 2, \end{align*} where $\boldsymbol e_l$ denotes the $l$-th standard basis vector. \par We use the notation $Y^R = \Re(Y)$ and $Y^I = \Im (Y)$ for any complex-valued tensor $Y$. \corr{This together with the differentiability of $\nonsep$ yields} the directional derivative in direction $\Y\in\SCK$ at $\B\in\SCK$ of the reduced functional \corr{$\nonsep$} \begin{align*} \del{\nonsep(\B)}{\B}[\Y] = \sum_{i=1}^K \langle \nabla ^{i,R} \nonsep(\B),(\Y)_i^R\rangle + \langle \nabla ^{i,I} \nonsep(\B),(\Y)_i^I\rangle. \end{align*} \subsection{Generalized Convex Approximation} We briefly recapitulate a \corr{number of} results from \cite{Stingl2009a}\corr{,} starting with the definition of a convex first-order approximation. \begin{definition}[convex first-order approximation]\label{defn:firstorder} We call an approximation $g:\SCK \to \R$ of a continuously differentiable function $\nonsep:\SCK \to \R$ a convex first-order approximation at $\bar \B =(\bar B_1,...,\bar B_\noDesign) \in \SCK$, if the following assumptions are satisfied \begin{enumerate}[label=\alph*),itemindent=0.5cm] \item $g(\bar \B) = \nonsep(\bar \B)$ \item $\del{g(\bar \B)}{B_i} = \del{\nonsep(\bar \B)}{B_i} $ for all $i\in \indexDesign$ \item $g$ is convex \end{enumerate} \end{definition} \begin{definition}[hyperbolic approximation]\label{defn:hyper} Let $\nonsep:\SCK\to \R$ be continuously differentiable on $\AdmSet\subset \SCK$ and $\bar {\B} = (\bar B_1,\ldots \bar B_\noDesign) \in \AdmSet$. Moreover\corr{,} let asymptotes $L=l\mathds1\in \SC$ and $U=u \mathds1\in \SC$ be given such that \begin{equation*} \Re(L)\prec \Re(Y) \prec \Re(U)\quad \text{and} \quad \Im(L) \prec \Im(Y) \prec \Im(U) \end{equation*} for all $Y \in \GraphSet$. Let $\tau$ \corr{be} a non-negative real parameter. Then we define the hyperbolic approximation $\hyperApprox$ of $\nonsep$ at $\bar {\B}$ as \begin{align}\label{eqn:hyper} \hyperApprox(\B) := & \nonsep(\bar {\B}) + \sum_{i=1}^\noDesign \partHyperApproxIR((\B)_i^R) -\constHyperApproxIR + \partHyperApproxII((\B)_i^I) -\constHyperApproxII \end{align} where the contributions of the real and imaginary part of $B$ are defined for $s\in\{R,I\}$ as \begin{align*} \partHyperApproxIs((\B)_i^s) & := \Big\langle (U^s - \bar B^s_i)\nabla ^{i,s}_+\nonsep(\bar {\B})(U^s - \bar B^s_i) + \tau (B^s_i-\bar B^s_i)^2, (U^s - B^s_i)^{-1}\Big\rangle \nonumber\\ & \qquad + \Big\langle (L^s - \bar B^s_i ) \nabla ^{i,s}_-\nonsep(\bar {\B})(L^s - \bar B^s_i) - \tau ((\B)^s_i-\bar B^s_i)^2, (L^s - (\B)^s_i)^{-1}\Big\rangle\nonumber \\ \constHyperApproxIs & := \Big\langle \nabla ^{i,s}_+\nonsep(\bar {\B}) , (U^s - \bar B^s_i)\Big\rangle + \Big\langle \nabla ^{i,s}_-\nonsep(\bar {\B}) , (L^s-\bar B^s_i)\Big\rangle\nonumber\corr{.} \end{align*} Here $\nabla^{i,s}_+\nonsep(\bar \B)$ and $\nabla ^{i,s}_-\nonsep(\bar \B)$ are the projections of $\nabla^{i,s}\nonsep(\bar \B)$ onto $\S_+$ and $\S_-$, \ie the space of symmetric positive definite and symmetric negative definite tensors, respectively. \end{definition} \begin{definition}[separable function on $\SCK$]\label{defn:seperable} A function $v:\SCK \to \R$ is called separable on $\SCK$ iff there exist $\bar \B\in \SCK$ and $\bar v_i : \SC \to \R$ for all $i \in\indexDesign$ such that \begin{equation*} v(\B) = v(\bar {\B}) + \sum_{i=1}^\noDesign \bar v_i((\B)_i) \quad \forall \B\in \SCK. \end{equation*} \end{definition} \begin{theorem} The hyperbolic approximation $\hyperApprox$ of $\nonsep$ given in \cref{defn:hyper} is a convex first-order approximation according to \cref{defn:firstorder} and separable on $\SCK$ according to \cref{defn:seperable}. \end{theorem} \begin{proof} The theorem is a straightforward extension of \cite[Theorem 3.4]{Stingl2009a} to the case of complex-valued material tensors. \end{proof} \begin{remark}[Proximal point terms]\label{rem:proxpoint} The terms $\tau \langle (B^s_i-\bar B^s_i)^2, (U^s - B^s_i)^{-1}\rangle$ and $\tau \langle ((\B)^s_i-\bar B^s_i)^2, (L^s - (\B)^s_i)^{-1}\rangle$ in \cref{defn:hyper} act as proximal point terms; accordingly the parameter $\tau$ takes the role of a proximal point parameter. As shown for the real-valued case in \cite{Stingl2009a}\corr{,} for $\tau > 0$ the hyperbolic approximations introduced above become uniformly convex with a modulus of the type $c_0+\tau c_1$, where $c_0\in \R_{\geq 0}$ and $c_1\in \R_{>0}$ are appropriate constants. \end{remark}\par To establish a solution scheme for \eqref{eqn:opt:probTMdisred}, we define the model problem \begin{equation}\label{eqn:Pjp} \begin{aligned} \min_{\B\in\AdmSet} ~ &\totalHyperApprox(\B) \end{aligned} \tag{\ensuremath{P^{\tau,\gamma}_{h,\bar\B}}} \end{equation} with the objective functional $\totalHyperApprox(\B) := \hyperApprox(\B)+\gamma \Jgraydisc(\B)$, i.e. we have applied the hyperbolic approximation \eqref{defn:hyper} to the non-separable functional $\nonsep$. \par We are now in the position to state the so-called sequential global programming algorithm, see \cref{alg:SGP}. We note that in each major iteration\corr{,} a finite number of sub-problems of type \cref{eqn:Pjp} \corr{are} solved. The inner loop realizes a globalization strategy: whenever the solution of the sub-problem does not provide sufficient descent for the objective of the original problem, the proximal point parameter $\tau$ is increased. Of course, in practice the stopping criterion of the outer loop is relaxed by a small positive constant. \begin{algorithm}\caption{Sequential Global Programming}\label{alg:SGP} \begin{algorithmic} \State Choose $\B^1\in \AdmSet$, $j \gets 1$, $\theta >1$, $\delta>0$, $\gamma\ge 0$ \Loop \State Choose $\tau$ \State $\bar \B \gets \B^j$ \Loop \State Solve \eqref{eqn:Pjp} $\rightarrow \B^+$ to global optimality \If{$\Jtotal(\B^+) < \totalHyperApprox(\B^j) - \delta \| \B^+ -\B^j\|_{F^\noDesign}^2$} \State $\B^{j+1} \gets \B^+$, {\bf break inner loop} \EndIf \State $\tau \gets \theta \tau$ \EndLoop \If {$\|\B^{j+1} -\B^{j}\|_{F^\noDesign}^2 = 0$} \State $\B^* \gets \B^{j+1}$, {\bf break outer loop} \EndIf \State $j \gets j+1$ \EndLoop \end{algorithmic} \end{algorithm}  \subsection{Convergence theory} In order to be able to prove a global convergence result for \cref{alg:SGP}\corr{,} we need a few technical definitions \corr{and} assumptions. \corr{This is because standard optimality conditions do not apply due to the potential non-smoothness of the grayness term and the structure of the feasible set in \cref{eqn:opt:probTMdis}}. \begin{definition}[tangential cone]\label{def:conv:tangentialcone} Let $l\in\indexEdges$ and $Y\in \GraphSet$, $\B\in \AdmSet$, then we define the following tangential cones: \begin{align*} \TangCone^{l}_Y & := \Big\{ D\in \SC \mid \exists (E^n)_{n\in\N} \subset \EdgeImage_l, (t^n)_{n\in\N} \subset \R_{>0},E^n \rightarrow Y, t^n\searrow 0 :\frac{E^{n}-Y}{t^n} \rightarrow D \Big\},\\ \TangCone^{l}_\B & := \Big\{ \D\in \SCK \mid (\D)_k \in \TangCone^l_{(\B)_k}, \forall k\in \indexDesign\Big\}. \end{align*} \end{definition} \begin{ass}[Assumptions on $\Jtotal$ and $\totalHyperApprox$]~\label{ass:conv} The functions $\Jtotal : \AdmSet \to \R$ and $\totalHyperApprox : \AdmSet \to \R$ are twice continuously \corr{directionally} differentiable at any $\B \in \AdmSet$ in any direction of the tangential cone $\TangCone^l_{\B}$ and for all $\bar \B \in \AdmSet$. \end{ass} We note that it is a simple exercise to show that the assumption on $\totalHyperApprox$ is satisfied for the grayness term and parametrizations discussed later in this section. Moreover\corr{,} the smoothness assumption on $\Jtotal$ is satisfied if the physical objective functional $\Jphysdisc$ is twice continuously differentiable. \begin{remark}[Inequalities for the objective functional and its approximation] Due to \cref{ass:conv} and the compactness of the set of admissible tensors $\GraphSet$ as well as the properties of the proximal point term discussed in \cref{rem:proxpoint}\corr{,} there exist $c_0\in \R_{\geq 0}$ and $c_1,c_2,c_3,c_4 \in \R_{>0}$ s.t. the following holds: \begin{align} \label{eqn:modulus} \secdel{\totalHyperApprox}{(\B)_k}[Y,Z] &\geq (c_0 + c_1 \tau)\|Y\|_F\|Z\|_F & &\forall \B \in \AdmSet, \forall Y,Z \in \TangCone_{(\B)_k}^l\\ \Big| \del{\Jtotal(\B)}{(\B)_k}[Y] \Big| & \leq c_2\|Y\|_F & &\forall \B \in \AdmSet, \forall Y \in \TangCone_{(\B)_k}^l \label{eqn:d_obj_func_lipschitz}\\ \label{eqn:dd_obj_func_lipschitz} \Big| \secdel{\Jtotal(\B)}{(\B)_k}[Y,Z] \Big| & \leq c_3\|Y\|_F\|Z\|_F & &\forall \B \in \AdmSet, \forall Y,Z \in \TangCone_{(\B)_k}^l \\ |\totalHyperApprox( \B) - \Jtotal( \B)| & \leq (c_4 + c_1\tau)\|\bar \B - \B\|_{F^\noDesign} ^2 & & \forall \B \in \AdmSet \label{eqn:approx_ineq}\\ \label{eqn:FOA_prop} \del{\totalHyperApprox(\B)}{(\B)_k}[Y] & = \del{\Jtotal(\B)}{(\B)_k}[Y] & & \forall Y \in \TangCone_{(\B)_k}^l \end{align} for all $(l,k,\tau) \in \indexEdges\times\indexDesign\times\R_{\geq 0}$ and $ \bar \B \in \AdmSet$\corr{.} \end{remark} Thus we can prove the following: \begin{lemma}[finite number of inner iterations] \label{lem:fnii} The inner loop terminates after a finite number of iterations. \end{lemma} \begin{proof} There exists $\tau_{\max}\in \R$ s.t. \begin{align*} \totalHyperApprox(\B) \geq \Jtotal(\B) \qquad \forall \B, \bar \B \in \AdmSet,\tau > \tau_{\max} \end{align*} and thus for $\hat \tau_{\max} := \tau_{\max} + \delta$ the stopping criterion of the inner loop of \cref{alg:SGP} is fulfilled for all $\tau > \hat \tau_{\max}$. This directly results from properties a) and b) in \cref{defn:firstorder} which hold for $\totalHyperApprox(\B)$, boundedness of the second derivative of $\Jtotal$ on the compact set $\AdmSet$ and the lower bound on the second derivative of $\totalHyperApprox$ depending on $\tau$ obtained from \cref{eqn:modulus}, \ie $\tau_{\max} = (c_3-c_0)(c_1)^{-1}$. \end{proof} \begin{theorem}[convergence results] Let $(\tau^m,\B^m)_{m\in\N}$ be a sequence generated by \cref{alg:SGP} and let \cref{ass:conv} be satisfied. Then \corr{there} exists $\Jtotal^* \in \R$ and $\B^* \in \AdmSet$ such that the following holds: \begin{enumerate} \item[a)] convergence of function values: $\quad \Jtotal(\B^m) \rightarrow \Jtotal^*$ \item[b)] convergence of material tensors: \begin{itemize} \item For $\delta > 0$: $\quad \B^m \rightarrow \B^* $ \item For $\delta = 0$: $\quad \B^{m_n} \rightarrow \B^* \quad$ for subsequence $(m_n)_{n\in\N}$\corr{.} \end{itemize} \end{enumerate} \end{theorem} \begin{proof}~ \begin{enumerate} \item[a)] Due to the monotonicity of $\Jtotal(\B^m)$, the continuity of $\Jtotal$ and the boundedness of the admissible set $\AdmSet$, we \corr{clearly have} convergence of the function values. \item[b)] The boundedness of $\AdmSet$ leads directly to the existence of a convergent subsequence. If $\delta > 0$\corr{,} then the stopping criterion of the inner loop of \cref{alg:SGP} leads to $\Jtotal(\B^{m+1}) \leq \Jtotal(\B^{m}) - \delta\|\B^m - \B^{m+1}\|^2_{F^\noDesign}$ and implies $\big|\Jtotal(\B^{m})-\Jtotal(\B^{m+1})\big| \geq \delta\|\B^m - \B^{m+1}\|^2_{F^\noDesign}$. Together with the convergence of the objective functional values, \ie $|\Jtotal(\B^m) - \Jtotal(\B^{m+1})| \rightarrow 0$\corr{,} we get \begin{align*} \|\B^m - \B^{m+1}\|^2_{F^\noDesign} \rightarrow 0. \end{align*} Together with the convergence along the subsequence to $\B^*$ we \corr{obtain} convergence of the whole sequence $(\B^m)_{m\in\N}$. \end{enumerate} \end{proof} Now, first order optimality conditions based on the tangential cone defined in \cref{def:conv:tangentialcone} can be written as: \begin{definition}[first order optimality]\label{def:firstorder} A material distribution $\B\in\AdmSet$ is called first order optimal to $\Jtotal$ iff \begin{align*} \del{\Jtotal(\B)}{(\B)_k}[(\Y)_k] \geq 0 \qquad \forall (l,k)\in \indexEdges \times \indexDesign, \Y \in \TangCone^l_\B\corr{.} \end{align*} \end{definition} \begin{theorem}[first order optimality of $\B^*$] Any accumulation point $\B^*$ of the sequence $(\B^k)_{k\in\mathds N}$ generated by \cref{alg:SGP} is first order optimal to $\Jtotal$. \end{theorem} \begin{proof} We will argue by contradiction \corr{and assume} $\B^*$ is not first order optimal. Then there exists an element index $k\in \indexDesign$, edge index $l\in \indexEdges$ and $(\Y)_k \in \TangCone^l_{(\B^*)_k}$ such that \begin{align} \del{\Jtotal(\B^*)}{(\B^*)_k}[(\Y)_k] = -\FOOptConst < 0. \label{eqn:LocDer} \end{align} Thus there exist sequences $(E^n)_{n\in\mathds N}$ and $(t^n)_{n\in\mathds N}$ as in \cref{def:conv:tangentialcone} of the tangential cone $\TangCone^l_{(\B^*)_k}$ with $E^n \rightarrow (\B^*)_k$ and $N_1\in \mathds N$ large enough that the following holds: \begin{align*} \del{\Jtotal(\B^*)}{(\B^*)_k}[E^n - (\B^*)_k] - \del{\Jtotal(\B^*)}{(\B^*)_k}[(\Y)_k] t^n & < \tfrac12 t^n \FOOptConst & \forall n > N_1\corr{.} \end{align*} Together with \eqref{eqn:LocDer}\corr{, we obtain}: \begin{align} \del{\Jtotal(\B^*)}{(\B^*)_k}[E^n - (\B^*)_k] &< -\tfrac12 t^n \FOOptConst & \forall n> N_1 \label{eq:approx_prop_6}. \end{align} Note that the existence of $N_1$ fulfilling this inequality is given by the properties of the tangential cone. To continue the proof, we derive an estimate for the directional derivative in \corr{the} direction of $E^n$ for a small change in the expansion point $\B^m$, \ie \begin{align*} &\Big|\del{\Jtotal(\B^m)}{(\B^m)_k} [E^n - (\B^m)_k] - \del{\Jtotal(\B^*)}{(\B^*)_k} [E^n - (\B^*)_k]\Big|\\ & \quad = \Big|\del{\Jtotal(\B^m)}{(\B^m)_k} [E^n - (\B^m)_k + (\B^*)_k - (\B^*)_k] \\ & \qquad - \del{ \Jtotal(\B^*)}{(\B^*)_k} [E^n - (\B^m)_k + (\B^*)_k - (\B^*)_k] - \del{\Jtotal(\B^*)}{(\B^*)_k} [(\B^m)_k-(\B^*)_k]\Big|\\ & \quad \leq c_2\|\B^* - \B^m\|_{F^\noDesign} + c_3\|\B^*-\B^m\|_{F^\noDesign} \Big(\|\E^n - \B^*\|_{F^\noDesign} + \|\B^*-\B^m\|_{F^\noDesign} \Big) \end{align*} where we have used \cref{eqn:d_obj_func_lipschitz} and \cref{eqn:dd_obj_func_lipschitz}. As $\B^m \rightarrow \B^*$ and $n\in \N$ is fixed, \corr{based on the latter estimate we can choose} $M_1(t^n)\in\N$ sufficiently large s.t. \begin{align} \Big|\del{\Jtotal(\B^m)}{(\B^m)_k} [E^n - (\B^m)_k] - \del{\Jtotal(\B^*)}{(\B^*)_k} [E^n - (\B^*)_k]\Big| < \tfrac14t^n\FOOptConst\quad \forall m> M_1(t^n) \label{eq:approx_prop_10}. \end{align} If $\B^m$ converges only along a sub-sequence, we redefine $(\B^m)_{m\in\N}$ by a sub-sequence thereof, \ie $\B^m := \B^{m_n}$ and continue with the same arguments. Thus by \eqref{eq:approx_prop_6} and \eqref{eq:approx_prop_10}, for a slight change of the expansion point the directional derivative \corr{remains} negative, \ie \begin{align} \del{\Jtotal(\B^m)}{(\B^m)_k}[E^n - (\B^m)_k] &< -\tfrac14t^n\FOOptConst & \forall m > M_1(t^n), \forall n>N_1\corr{.} \label{eqn:LocDer1} \end{align} Moreover, by \eqref{eqn:FOA_prop} and \eqref{eqn:approx_ineq} the following holds for the first order accurate model $\totalHyperApproxArg{\tau}{\B^m}$ with $\E^n:= (B^m_1,\dotsc,B^m_{k-1},E^n,B^m_{k+1},\dotsc,B^m_K)$ for all $(m,n,\tau) \in \N\times \indexDesign \times \R_{>0}$: \begin{align} \totalHyperApproxArg{\tau}{\B^m}(\E^n) & \leq \Jtotal(\B^m) + \del{\Jtotal(\B^m)}{(\B^m)_k}[E^n - (\B^m)_k] + (c_4 + c_1\tau)\|E^n - (\B^m)_k\|_F^2. \label{eqn:mod_prop} \end{align} As $E^n$ \corr{converges} to $(\B^*)_k$, there exists $N_2 > N_1$ such that \begin{align*} \|E^n - (\B^*)_k\|^2_F &\leq \frac{t^n \FOOptConst}{8(c_4+c_1\overline{\tau})} & \forall n> \corr{N_2} \end{align*} with $\overline{\tau}$ being the maximal possible $\tau$ of \cref{alg:SGP}. This maximum exists as shown in \cref{lem:fnii} and is bounded by the maximum of $\theta\tau_{\max{}}$ and the initial value for $\tau$, where $\theta > 1$ is the scaling parameter within the inner loop of \cref{alg:SGP}. \par By inequality \eqref{eqn:LocDer1} and the latter estimate we \corr{have} from \cref{eqn:mod_prop} \begin{align} \label{eqn:DoNotKnow} \totalHyperApproxArg{\tau}{\B^m}(\E^n) & < \Jtotal(\B^m) - \tfrac18t^n\FOOptConst & \forall m > M_1(t^n), \forall n>N_2. \end{align} As $\B^m \rightarrow \B^*$ for $m\rightarrow \infty$ we can choose $M_2(t^n) > M_1(t^n)$ \corr{sufficiently} large s.t. \begin{align*} |\Jtotal(\B^m) - \Jtotal(\B^*)| &< \tfrac{1}{16} t^n \FOOptConst & \forall m > M_2(t^n) , \forall n \in \N. \end{align*} Combining the last two inequalities and noting that \cref{eqn:DoNotKnow} holds for all $\tau$ we can choose $n>N_2$ and $m>M_2(t^n)$ for which the following inequality is satisfied: \begin{align*} \totalHyperApproxArg{\tau_{m}}{\B^m}(\E^n) < \Jtotal(\B^*) - \tfrac{1}{16} t^n \FOOptConst. \end{align*} Here, \corr{$\tau_m$ denotes} the actual proximal point parameter used in the $m$-th outer iteration of \cref{alg:SGP}. Finally noting that $\B^{m+1}$ is the global minimizer of the sub-problem $\min_{\B\in\AdmSet} \totalHyperApproxArg{\tau_m}{\B^m}(\B)$ and taking the inner stopping criterion in \cref{alg:SGP} into \corr{consideration}, we arrive at \[\Jtotal(\B^{m+1}) \leq \totalHyperApproxArg{\tau_{m}}{\B^m}(\B^{m+1}) \leq \totalHyperApproxArg{\tau_{m}}{\B^m}(\E^n) < \Jtotal(\B^*) - \tfrac{1}{16}t^n\FOOptConst.\] This is in contradiction to the monotonicity properties of the sequence of objective function values, i.e. $(\Jtotal(\B^m))_{m\in\N}$. Thus, $\B^*$ is first order optimal in the sense of \cref{def:firstorder}. \end{proof} \subsection{Solution of the subproblem}\label{sec:subprobsolve} In this section \corr{we describe how} the sub-problems in \cref{alg:SGP} can be efficiently solved. In order to do this, we fix the definition of the grayness functional and the parametrization of the feasible set. In this paper we restrict ourselves to \corr{a} grayness function of the following type: \begin{definition}[grayness penalization on the graph $\AdmSet$]\label{defn:grayfun} On $\AdmSet$ we define for all $\B = (B_1,\dotsc,B_\noDesign)\in \AdmSet$ a grayness penalization $\Jgraydisc$ by \begin{align*} \Jgraydisc(\B) & = \sum_{i=1}^K \tilde \Jgraydisc(B_i) & \text{with}&& \tilde \Jgraydisc(B_i) & := \sum_{l\in \indexEdges }\begin{cases} \Param_l^\mo (B_i) (1-\Param_l^\mo (B_i) ) & \text{if } B_i \in \EdgeImage_l\setminus\NodeImage \\ 0 & \text{otherwise}. \end{cases} \end{align*} \end{definition} We note that for parametrizations $(\psi_l)_{l\in \noEdges}$ satisfying the assumptions in \cref{ass:properties}, the directional differentiability on each edge of the parametrization required in \cref{ass:conv} is satisfied for the grayness functional stated in \cref{defn:grayfun}. Next, we reformulate \eqref{eqn:Pjp} in terms of the parametrization $\Param$: \begin{equation}\label{eqn:QC}\tag{\text{$Q^\gamma$}} \begin{aligned} \min_{\b l \in(\indexEdges)^\noDesign} \min_{\b \alpha\in[0,1]^\noDesign} & \ \hyperApprox (\Param_{l_1}(\alpha_1),\dotsc,\Param_{l_\noDesign}(\alpha_\noDesign)) + \gamma \sum_{i=1}^\noDesign \alpha_i (1- \alpha_i)\corr{.} \end{aligned} \end{equation} \begin{lemma} If $( \b \alpha^*,\b l^*)$ is a global optimal solution of \eqref{eqn:QC}, then $(\B)_i = \Param_{l_i^*}(\alpha_i^*)$ for all $i \in \indexDesign$ is a global optimal solution of \eqref{eqn:Pjp}. \end{lemma} \begin{proof} Since $\tilde \Jgraydisc(\Param_{l_i}(\alpha_i)) = \alpha_i(1-\alpha_i)$ for all $i\in \indexDesign$ by \cref{defn:grayfun}, \eqref{eqn:QC} is a reparametrization of \eqref{eqn:Pjp}. That is why the global optimal solutions coincide. \end{proof} Due to the separability of $\hyperApprox$ and $\Jgraydisc$, we find the global optimum for \eqref{eqn:QC} if we find the global optimal solution of \begin{equation}\label{eqn:Qil}\tag{\text{$Q^i_l$}} \begin{aligned} \min_{\alpha_i \in[0,1]} & \ \partHyperApproxIR (\Re(\Param_{l}(\alpha_i))) + \partHyperApproxII (\Im(\Param_{l}(\alpha_i)))+ \gamma \alpha_i (1-\alpha_i) \\ \end{aligned} \end{equation} for each element $i\in \indexDesign$ and edge index $l\in \indexEdges$ (see \cref{alg:submultigamm}). \corr{Note that constant terms with respect to $\b\alpha$ are neglected here}. \begin{algorithm}[ht]\caption{solution of \eqref{eqn:Pjp}}\label{alg:submultigamm} \begin{algorithmic}[1] \For {$i\in \indexDesign$} \For {$l \in \indexEdges$} \State $\beta\gets$ solve \eqref{eqn:Qil} globally \State $(\b j^*)_l \gets \partHyperApproxIR (\Re(\Param_{l}(\beta))) + \partHyperApproxII (\Im(\Param_{l}(\beta)))+ \gamma \beta (1-\beta) $ \State $(\b \alpha^*)_l \gets \beta$ \EndFor \State Find index $l^*$ s.t. $(\b j^*)_{l^*} \le (\b j^*)_{l}$ for all $l \in \indexEdges $ \State $B_i \gets \Param_{l^*}((\b \alpha^*)_{l^*})$ \EndFor \end{algorithmic} \end{algorithm} The following result is useful when the desired solution is of \corr{a} discrete nature, \ie $(\B^*)_i \in \NodeImage, \forall i\in \indexDesign$. \begin{remark}\label{rem:discreteSol} For every $\tau>0$ there exists a $\gamma_{\text{max}}>0$ sufficiently large such that for all $\gamma> \gamma_{\max}$ the global optimal solution $(\b l^*_\gamma,\b\alpha^*_\gamma)$ satisfies $\Param((l_\gamma^*)_i,(\alpha_\gamma^*)_i) \in \NodeImage$ for all $i \in \indexDesign$. \end{remark} \begin{proof} Let $v_i((\B)_i) := \partHyperApproxIR( (\B)_i^R)+ \partHyperApproxII((\B)_i^I)$. Since $v_i\circ \Param_{l} \in C^2([0,1])$ for all $l\in \indexEdges$ and $i\in\indexDesign$, its second derivative is bounded from above by $\sigma \in \R$. By choosing $\gamma_\text{max} = \tfrac\sigma2$ the second derivative of $\delta \mapsto v_i(\Param_{l} (\delta)) + \gamma \delta (1-\delta) $ is strictly negative for all $\gamma > \gamma_\text{max}$ for all $l\in \indexEdges$ and $i\in\indexDesign$. Thus the second order optimality conditions are never fulfilled and the global minimum of the latter function restricted to $[0,1]$ is located on the boundary. \end{proof} In the next section, we \corr{provide} strategies to obtain the global optimal solution of \eqref{eqn:Qil} for two particular choices of parametrizations. Moreover, to shorten the notation we define for $s\in\{R,I\}$ \begin{align}\label{eqn:derabbrev} \constParamApproxIsU & := (U^s - \bar B^s_i) \nabla ^{i,s}_+\nonsep(\bar {\B}) (U^s - \bar B^s_i), & \constParamApproxIsL & := (L^s - \bar B^s_i ) \nabla ^{i,s}_-\nonsep(\bar {\B})(L^s - \bar B^s_i), \end{align} since these terms are independent of the design parameters in both cases. \subsubsection{Rotational parametrization}\label{sec:rotparam} \begin{definition}[rotational parametrization]\label{defn:rotparam} Let the so-called reference material tensor $B^{(r)}\in \SC$ be a diagonal tensor. We call a parametrization $\Param$ of material tensors \corr{a} rotational parametrization based on $B^{(r)}$, iff \begin{align*} \Param(\delta) = R(\pi \delta) B^{(r)} R(\pi \delta)^T \qquad \delta\in [0,1] \end{align*} with rotation matrix $R\colon\R \to \operatorname{SO}(2)$: \begin{align*} R(\pi\delta) = \begin{pmatrix} \cos(\pi\delta) & -\sin(\pi\delta)\\\sin(\pi\delta)& \cos(\pi\delta) \end{pmatrix}. \end{align*} \end{definition} \begin{theorem}[global solution, rotational parametrization] For a rotational parametrization $\Param_l$ based on a real-valued diagonal reference tensor $B^{(r)} \in \SR$ \corr{with} asymptotes satisfying the assumptions of \cref{defn:hyper} and $\gamma=0$, the parametrized subproblem \eqref{eqn:Qil} has the global optimal solution \begin{equation*} \alpha^*_i = \frac1{2\pi}\begin{cases} \arctan(-\frac{b}{a}) + \pi & a < 0\\ \arctan(-\frac{b}{a})\mod 2\pi & a > 0\\ \frac\pi2\sign(b)+\pi & a = 0 \end{cases} \end{equation*} with the two parameters \begin{align*} a & = \tau \big[(\bar B_i)_{22}-(\bar B_i)_{11}\big] c_0 + c_{1}, & b &= \tau \big[(\bar B_i)_{12}+(\bar B_i)_{21}\big] c_0 + c_{2}\corr{.} \end{align*} \corr{Here,} \begin{align*} c_0 & = 2(B^L_{11} - B^U_{11}) B^{(r)}_{11} + 2(B^U_{22} - B^L_{22}) B^{(r)}_{22} & \\ & \quad + \big[ (B^U_{11} - B^U_{22})+(B^L_{22}-B^L_{11})\big]\big[(\bar B_i)_{11}+(\bar B_i)_{22}\big]\corr{,}\\ c_{1} & = (B^U_{11}-B^U_{22}) \big[(C_U^{i,R})_{22}-(C_U^{i,R})_{11}\big] +(B^L_{11}-B^L_{22}) \big[(C_L^{i,R})_{22}-(C_L^{i,R})_{11}\big]\corr{,}\\ c_{2} & = (B^U_{11}-B^U_{22}) \big[(C_U^{i,R})_{12}+(C_U^{i,R})_{21}\big] + (B^L_{11}-B^L_{22}) \big[(C_L^{i,R})_{12}+(C_L^{i,R})_{21}\big]\corr{,} \end{align*} and we use the abbreviations $ B^L = (L-B^{(r)})^\mo$, $B^U = (U-B^{(r)})^\mo$. Furthermore, $\sign$ denotes the sign function and $\mod$ the modulo operation, respectively. \end{theorem} \begin{proof} First, we compute all stationary points of the hyperbolic approximation \eqref{eqn:hyper}, \corr{thus we} can neglect terms which are constant with respect to material tensors $\B = (B_1,\dotsc,B_\noDesign)$. For a fixed element $i\in \indexDesign$ we obtain with \eqref{eqn:derabbrev} \begin{align*} \partHyperApproxIR (B_i) = & \left\langle \constParamApproxIRU + \tau (B_i-\bar B_i)^2 ,(U - B_i)^\mo\right\rangle + \left\langle \constParamApproxIRL - \tau (B_i-\bar B_i)^2 , (L-B_i)^\mo\right\rangle. \end{align*} Note that $\constParamApproxIRU$ and $\constParamApproxIRL$ are independent on $B_i$. Taking the rotational parametrization into account, \ie $B_i = R(\pi \alpha_i) B^{(r)} R(\pi \alpha_i)^T$, using the choice of asymptotes $L = l \mathds1$ and $U = u \mathds1$ and by properties of the scalar product as well as rotation matrices this can be rewritten as \begin{align*} \mathrlap{ \partHyperApproxIR (R(\pi \alpha_i) B^{(r)} R(\pi \alpha_i)^T) =}\quad & \\ & \left\langle R(\pi \alpha_i)^T \constParamApproxIRU R(\pi \alpha_i) + \tau (B^{(r)} -R^T(\pi \alpha_i)\bar B_iR(\pi \alpha_i))^2 ,(U - B^{(r)} )^\mo\right\rangle \nonumber \\ & + \left\langle R(\pi \alpha_i)^T \constParamApproxIRL R(\pi \alpha_i) - \tau (B^{(r)} -R^T(\pi \alpha_i)\bar B_iR(\pi \alpha_i))^2 ,(L - B^{(r)} )^\mo\right\rangle.\nonumber \end{align*} After straightforward calculus using angle sum identities, the derivative of $\partHyperApproxIR(B_i)$ with respect to the design parameter $\alpha_i$ has the form \begin{align}\label{eqn:statioary} \tot{ \partHyperApproxIR (R(\pi \alpha_i) B^{(r)} R(\pi \alpha_i)^T) }{\alpha_i} & = a \pi\sin(2\pi \alpha_i) + b \pi \cos(2\pi \alpha_i) \end{align} with the coefficients $a$, $b$ as given in the theorem. The stationary points of $\partHyperApproxIR(B_i)$ are given by the two roots of \eqref{eqn:statioary} in the interval $[0,1)$. Now we choose the root for which the second derivative of $\partHyperApproxIR(B_i)$ with respect to $\alpha_i$ is positive. This root is given by the formula for $\alpha_i^*$ \corr{stated} in the theorem. \end{proof}  \subsubsection{Polynomial parametrization}\label{sec:polyparam} In this paragraph, we provide a solution scheme to solve the subproblem \eqref{eqn:QC} if the material tensor is parametrized by a polynomial on an edge of $\GraphSet$. We show that in this case the hyperbolic approximation $\hyperApprox$ is a rational polynomial and solve \eqref{eqn:Qil} by finding the roots of its derivative. \begin{definition}[polynomial parametrization] We call a parametrization $\Param$ of material tensors polynomial parametrization of order $k$, \ie $\Param \in \mathcal P ^k_{\SC}$, if \begin{align*} \Param(\delta) = B^{(1)}(1-\delta) + B^{(2)} \delta + \delta(1-\delta) \sum_{i=0}^{k-2} A_i \delta^i, \qquad \delta \in [0,1] \end{align*} with $B^{(1)}\in \SC$, $B^{(2)}\in \SC$ and interpolation coefficients $A_i\in \SC$, $0\le i \le k-2$. \end{definition} \begin{lemma} \label{lem:numden} Let the parametrization $\Param_l$ be given by a polynomial of order $k$ over ${\SC}$, \ie $\Param_l \in\mathcal P^k_{\SC} $. Then the contribution of the $i$-th element to the hyperbolic approximation \eqref{eqn:hyper}, for $Y:=\Param_l(\delta)$, is \begin{equation*} \partHyperApproxIR (Y^R)+ \partHyperApproxII (Y^I) = \frac{p^i_l(\delta)}{q^i_l(\delta)} \end{equation*} where \begin{equation*} \begin{aligned} p^i_l(\delta) & = \sum_{s\in\{R,I\}} \frac { N_L^{i,s}(\delta) q_l^i (\delta)}{ \det(L^s-Y^s)} + \frac{ N_U^{i,s}(\delta) q_l^i (\delta)}{\det(U^s - Y^s)} \in \mathcal P^{9k}_\R,\\ q^i_l(\delta) & = \prod_{s\in\{R,I\}}\det(L^s-Y^s) \det(U^s - Y^s) \in \mathcal P^{8k}_\R \end{aligned} \end{equation*} and \begin{align*} N^{i,s}_L(\delta) & = \langle \constParamApproxIsL - \tau (Y^s-\bar B_i^s)^2 , \adj(L^s- Y^s)\rangle \\ N^{i,s}_U(\delta) & = \langle \constParamApproxIsU + \tau (Y^s-\bar B_i^s)^2 , \adj(U^s - Y^s)\rangle. \end{align*} \end{lemma} \begin{proof} Together with the definitions for $\constParamApproxIsU$ and $\constParamApproxIsL$ from \eqref{eqn:derabbrev}, the contribution of the real and imaginary part of $Y:= \Param_l(\delta)$ to \eqref{eqn:hyper} is \begin{align*} \partHyperApproxIs(Y) = \langle \constParamApproxIsL - \tau (Y-\bar B_i^s)^2 , (L^s-Y)^\mo \rangle + \langle \constParamApproxIsU + \tau (Y-\bar B_i^s)^2 , ( U^s-Y)^\mo \rangle. \end{align*} for $s\in\{R,I\}$, respectively. Using the formula $A^{-1}\det(A) = \adj(A)$ for a matrix $A$, we \corr{obtain} \begin{align*} \partHyperApproxIs(Y) = & \det(L^s-Y)^\mo \langle \constParamApproxIsL - \tau (Y-\bar B_i^s)^2 , \adj(L^s-Y) \rangle \\ & + \det(U^s-Y)^\mo \langle \constParamApproxIsU + \tau (Y-\bar B_i^s)^2 , \adj(U^s-Y) \rangle. \end{align*} \corr{Using} the common denominator $q^i_l$, we \corr{have} the expression \begin{equation*} \partHyperApproxIR (Y^R)+ \partHyperApproxII (Y^I) = \frac{p^i_l(\delta)}{q^i_l(\delta)} \end{equation*} with the polynomials $p^i_l(\delta)$ and $q^i_l(\delta)$ \corr{as stated} in the lemma. \end{proof} \begin{algorithm}[ht]\caption{solution of \eqref{eqn:Qil} with polynomial parametrization}\label{alg:polysolve} \begin{algorithmic}[1] \State Find the $n$ real roots $\alpha^{(1)},\dotsc,\alpha^{(n)}$ in the interval $(0,1)$ of the polynomial \begin{equation*} (q^i_l)'(\delta) p^i_l(\delta) - q^i_l(\delta) (p^i_l)'(\delta) + \gamma (1-2\delta) q^i_l(\delta)^2 \end{equation*}\vspace*{-\baselineskip} \State $\b \beta \gets (0,\alpha^{(1)},\dotsc,\alpha^{(n)},1) $ \For{$1\le k \le n+2$} \State $j_k \gets\partHyperApproxIR (\Re(\Param_l(\beta_{k}))) + \partHyperApproxII (\Im(\Param_l(\beta_{k}))) + \gamma \beta_k(1-\beta_k)$ \EndFor \State Find index $n^*$ s.t. $j_{n^*} \le j_k$ for all $1\le k \le n+2 $ \State $\alpha^*_{l} \gets \beta_{n^*}$ \end{algorithmic} \end{algorithm} \begin{theorem} With the assumptions of \cref{defn:hyper} and \cref{lem:numden}, \cref{alg:polysolve} yields a global optimal solution $\alpha^*_l$ of \eqref{eqn:Qil}. \end{theorem} \begin{proof} By \cref{lem:numden} the objective functional of \eqref{eqn:Qil} can be expressed as \begin{align*} j^i_l(\delta) := \frac{p^i_l(\delta)}{q^i_l(\delta)} + \gamma \delta (1-\delta). \end{align*} Necessarily, the global minimum of \eqref{eqn:Qil} is located either at $0$, $1$ or a root of the derivative of $j^i_l$. By \cref{defn:hyper}, the denominator of $j^i_l$ has no roots in the interval $(0,1)$, thus the roots of the derivative of $j^i_l$ coincide with the roots of its numerator. The global minimum of \eqref{eqn:Qil} is then selected by comparing the objective functional values for all candidates\corr{,} including the boundary points $0$ and $1$. \end{proof} We finally note that under specific assumptions on the admissible material tensors, the degree of the polynomial of Step 1 of \cref{alg:polysolve} can be significantly reduced. For instance, in the case of isotropic, real-valued materials and linear interpolation the roots of a cubic polynomial \corr{must} be computed. Moreover\corr{,} in the general case\corr{,} to find the roots of a normalized polynomial, one possible approach is to compute eigenvalues of its companion matrix \cite{Edelman1995}.  \section{Examples}\label{sec:example}~ In the following we discuss two different examples \corr{to illustrate this approach}. The purpose of the first example is twofold: First we \corr{wish} to investigate \corr{the performance of \cref{alg:SGP}} for optimization problems involving arbitrarily oriented anisotropic materials. In a second step, we want to examine \corr{the performance of the algorithm} when only a finite subset of orientations is admissible. In particular, we \corr{wish} to study \corr{the extent to which} the quality of the locally optimal solutions \corr{depends} on the finite number of orientations. This is important because as a consequence of \cref{rem:discreteSol} and \cref{lem:fnii} it is clear that for sufficiently large $\gamma$ every element of the set $\NodeImage^\noDesign$ is a local minimum of problem \cref{eqn:opt:probTMdisred}. \par The second example \corr{demonstrates} the capabilities of \cref{alg:SGP} \corr{when} the set of admissible material is parametrized by a complete graph with given complex-valued and isotropic material tensors at the nodes. \corr{To achieve this}, a material distribution is reconstructed by the information covered in the scattered magnetic field. Furthermore the effect of the regularization is investigated.  \subsection{Cloaking of a scatterer} \begin{figure}\centering \scalebox{0.7}{ \includegraphics[]{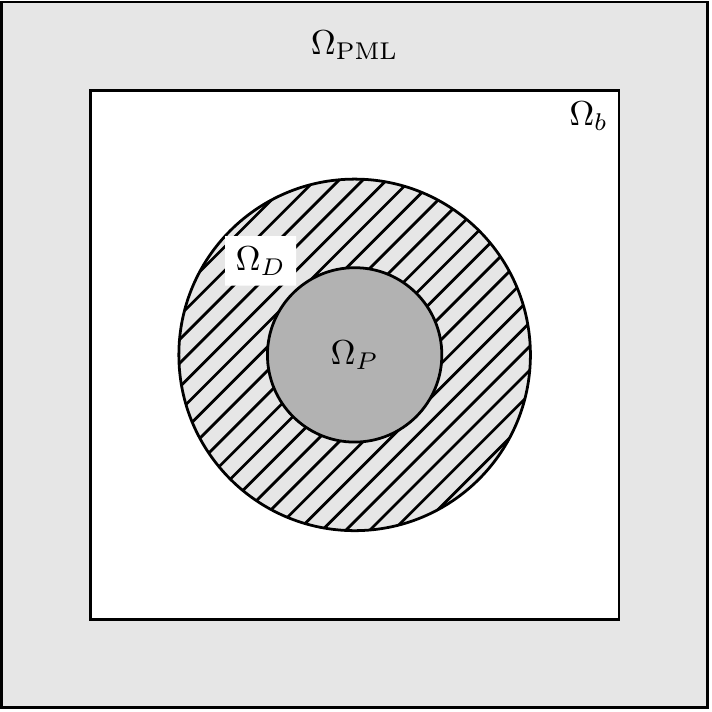} } \caption{Cloaking of a scatterer: Optimization setting}\label{fig:CSopt} \end{figure} The first example aims \corr{to minimize the} visibility of an absorbing core particle $\Omega_P$ \corr{using} optimal local orientation of an anisotropic material in the design domain $\Omega_D$ surrounding the particle (see \cref{fig:CSopt}). We assume that the background material tensor $B_b$ is real valued. \corr{The} visibility of a scattering object can be quantified by the amount of absorbed energy and scattered energy \cite{Mishchenko2014}. The absorbed energy is related to the absorption cross section which is defined as \begin{equation*} W^{\text{abs}} := -\int_{\d U} \frac12 \Re \left( E_T \times H_T^*\right)\cdot n \rmd \omega, \end{equation*} \ie the total energy flow, depending on the total electric $E_T$ and total magnetic field $H_T$, through the boundary of a neighborhood $U$ of the whole scatterer $\Omega_D \cup \Omega_P$, which can be chosen, for instance, as a \corr{ball}. \corr{The} scattered energy is proportional to the scattering cross section which reads \begin{equation*} W^{\text{sca}} := \int_{\d U} \frac12 \Re \left( E \times H^*\right)\cdot n \rmd \omega \end{equation*} with scattered field quantities $E$ and $H$. Note that both values, $W^{\text{ext}}$ and $W^{\text{sca}}$ are independent \corr{of} the choice of $U$ and non-negative. Absorption and scattering can be combined \corr{into} a common value\corr{,} which is called \corr{the} extinction cross section \begin{equation*} W^{\text{ext}} := W^{\text{abs}} + W^{\text{sca}} = \int_{\d U} \frac12 \Re \left( E_S \times H_I^* + E_I \times H_S^*\right)\cdot n \rmd \omega\corr{.} \end{equation*} \corr{This will serve} as the physical objective functional $\Jphys$ in this example after a \corr{number of} further adaptions. First we introduce the tensor-valued function $B_\Omega:\Omega\to \SC$ with \begin{equation*}B_\Omega(x) = \begin{cases} B_C(x) & x \in \Omega_C\\ B(x) & x\in \Omega_D. \end{cases} \end{equation*} Using transverse magnetic assumptions and switching to the two-dimensional setting, the extinction cross section transforms to \begin{align*} \Jphys(B,u) = -\Re \left( \frac\imath{2\omega} \int_{\d U} ( u^*_I (B_\Omega\nabla u + (B_\Omega-B_b)\nabla u_I) + u^* B_b\nabla u_I )\cdot n \rmd \omega \right). \end{align*} Moreover with partial integration and \eqref{eqn:weakbilinear}, we can proceed to a volume integral instead of the boundary integral and arrive at the objective functional in \corr{its} final form: \begin{align}\label{eqn:objTM} \Jphys(B,u) & = - \Re \left (\frac\imath{2\omega} \int_{U} \nabla u_I^H (B_\Omega - B_b) (\nabla u + \nabla u_I) \rmd x\right). \end{align} Note that the integral in the objective functional \eqref{eqn:objTM} can be restricted to $\Omega_D \cup \Omega_C$.  \subsubsection{Optimization problem} We collect the results from the previous sections and give the full optimization problem for this example: \begin{align*} \min_{\B} ~ & \Jphysdisc(\B,\u) + \eta \Jregdisc(\B) + \gamma \Jgraydisc(\B) \\ \st ~ & (\SYS(\B) + \SYS_C) \u = \RHS(\B) + \RHS_C \\ & \corr{\text{where }}\B \text{ is parametrized by rotation angles} \end{align*} with the objective function $\Jphysdisc(\B ,\u) = \Re ( \u^H (V(\B)+V_P) + W(\B) + W_P)$ based on \eqref{eqn:objTM}. The vectors $V(\B),V_P\in \C^\noDoF$ are \corr{defined} element-wise as for all $1\le i \le \noDoF$ \begin{align*} (V(\B))_i & = \frac\imath{2\omega} \sum_{k=1}^{\noDesign} \int_{\TriI k} \nabla u_I^T (B_k^* - B_b^*) \nabla \phi_i \rmd x \\ (V_P)_i & = \frac\imath{2\omega} \int_{\Omega_P} \nabla u_I^T (B_P^* - B_b) \nabla \phi_i \rmd x \\ \intertext{and the scalars $W(\B),W_P\in \C$ are } W(\B) & = \frac\imath{2\omega} \sum_{k=1}^{\noDesign} \int_{\TriI k} \nabla u_I^T (B_k^* - B_b^*) \nabla u_I^* \rmd x \\ W_P & = \frac\imath{2\omega} \int_{\Omega_P} \nabla u_I^T (B_P^* - B^*_b) \nabla u_I^* \rmd x\corr{.} \end{align*} Moreover\corr{,} $\SYS,\SYS_C,\RHS$ and $\RHS_C$ are given in \cref{sec:discrete}. Together with the adjoint variable $\b p$ which solves the adjoint equation \begin{align*} (\SYS(\B) + \SYS_C) \b p = - (V(\B) + V_C)^* \end{align*} we \corr{obtain} the derivative of the physical objective functional in direction $\Y\in \SCK$ for all $k\in \indexDesign$ \begin{align*} \tot{\Jphysdisc(\B, \u(\B))}{(\B)_k}[(\Y)_k] & = \Re\left ( \int_{\TriI k} \Big ( \nabla_{\!\!h} \b p - \frac\imath{2\omega} \nabla u_I^* \Big)^T (\Y)_k \Big(\nabla_{\!\!h} \u + \nabla u_I\Big) \rmd x \right ). \end{align*} Here, $\nabla_{\!\!h} \b u = \sum_{i=1}^\noDoF u_i\nabla \phi_i$ and $\nabla_{\!\!h} \b p = \sum_{i=1}^\noDoF p_i\nabla \phi_i$. Based on this, the sub-problems in \cref{alg:SGP} can be established and solved \corr{using the} strategies discussed in \cref{sec:rotparam}.  \subsubsection{Numerical results} The particle domain $\Omega_P$ with material tensor $B_P = (0.1 + 2\imath)^{-2}\mathds1$ is a \corr{ball} with radius 0.2, see \cref{fig:CSopt}. Inside the design domain $\Omega_D$, which is a \corr{ball} with radius 0.4, the material tensor is optimized with rotational parametrization based on the diagonal reference material tensor $B^{(r)} = \diag(1,2)^{-2}$. A box with a side length of 2 defines $\Omega_b$ and a layer with thickness 1 represents the PML both with material properties $B_b=\mathds1$. \par Furthermore, we choose a plane incident wave $u_I(x,y)= \sqrt2\exp(\imath \omega x)$ with wavelength $\lambda = 0.55,\omega = \frac{2\pi}\lambda$ and set the PML function to $s(t) = 1 - \frac{100}{\imath\omega}\max(0,|t|-1)$. \par The state and adjoint equation are solved \corr{using} the Finite Element Method (FEM) on triangluar cells and implemented in MATLAB \cite{Matlab2014}. We use linear Lagrange basis functions \cite{Zienkiewicz2013} to approximate the scalar fields. The triangulation of the computational domain is generated with the Delaunay triangulation tool Triangle \cite{Shewchuk1996} and \corr{the} triangulation process provides approx. $\num[scientific-notation = fixed]{1e5}$ triangular elements in total and approx. $\num[scientific-notation = fixed]{3e4}$ triangles in the design region $\Omega_D$. \par Since in this example we consider a rotational parametrization (\cref{defn:rotparam}), the admissible set is given as $\AdmSet=\GraphSet^\noDesign$ with $\GraphSet = \{ R(\pi\delta) B^{(r)} R(\pi\delta)^T\mid \delta\in[0,1]\}$. Thus the underlying graph has only one closed edge. The asymptotes for \cref{alg:SGP} are defined by $l=0$ and $u=100$ and thus satisfy the assumptions in \cref{defn:hyper}. We choose the regularization parametrer $\eta=100$, the circular filter radius $r_0 = 0.01$ and the grayness penalty factor $\gamma=0$. \begin{figure}\centering \parbox[b]{0.5cm}{ \includegraphics[height=0.155\textwidth]{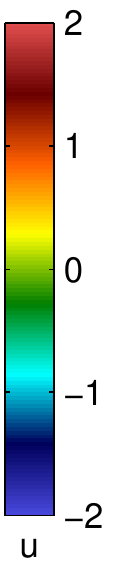}\\\includegraphics[height=0.155\textwidth]{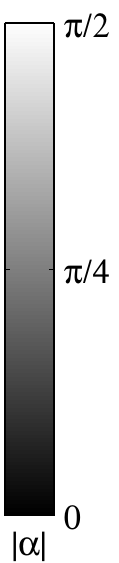}} \subfigure[inital\label{subfig:clk:inital}]{{\includegraphics[width=0.31\textwidth]{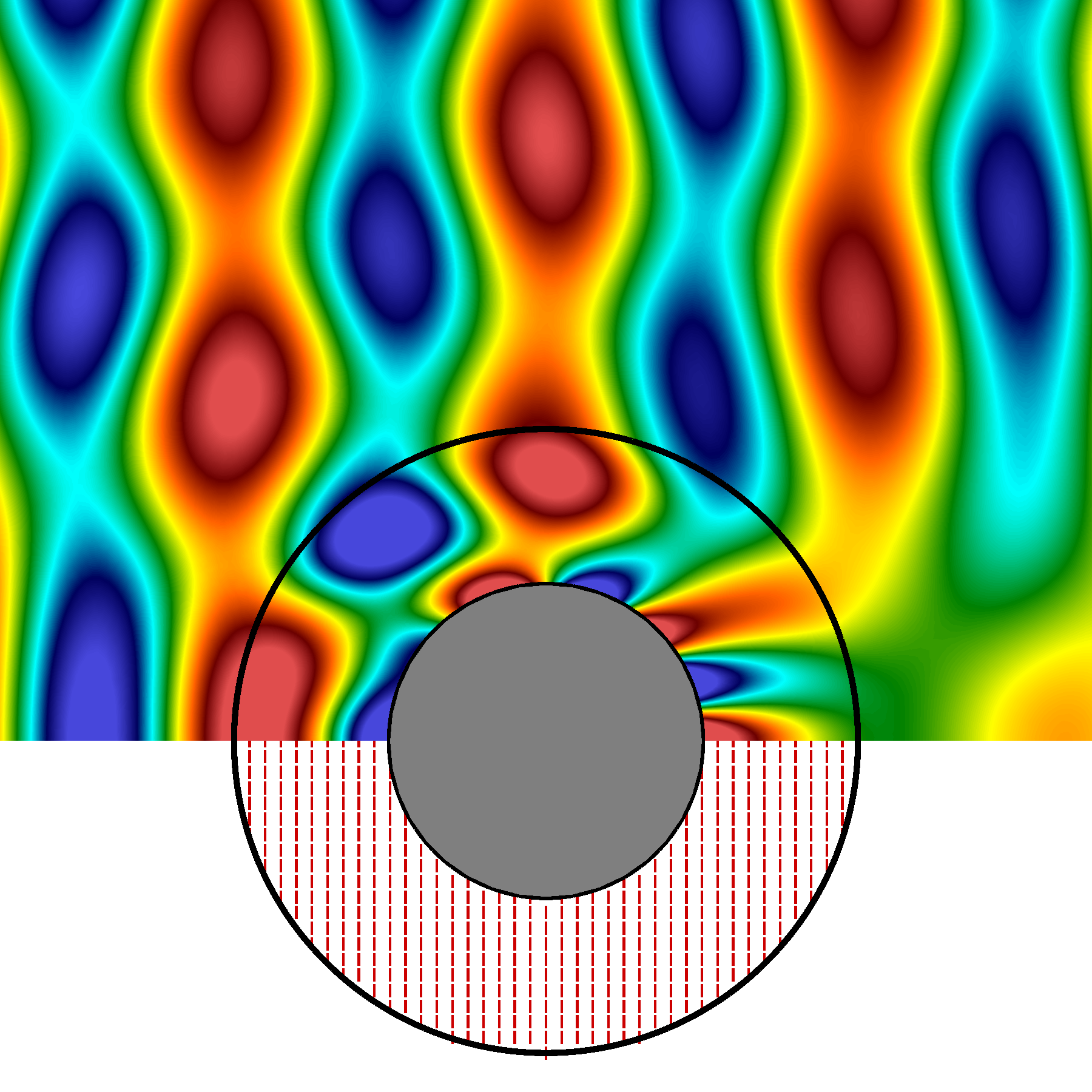}}} \subfigure[optimized\label{subfig:clk:final}]{{\includegraphics[width=0.31\textwidth]{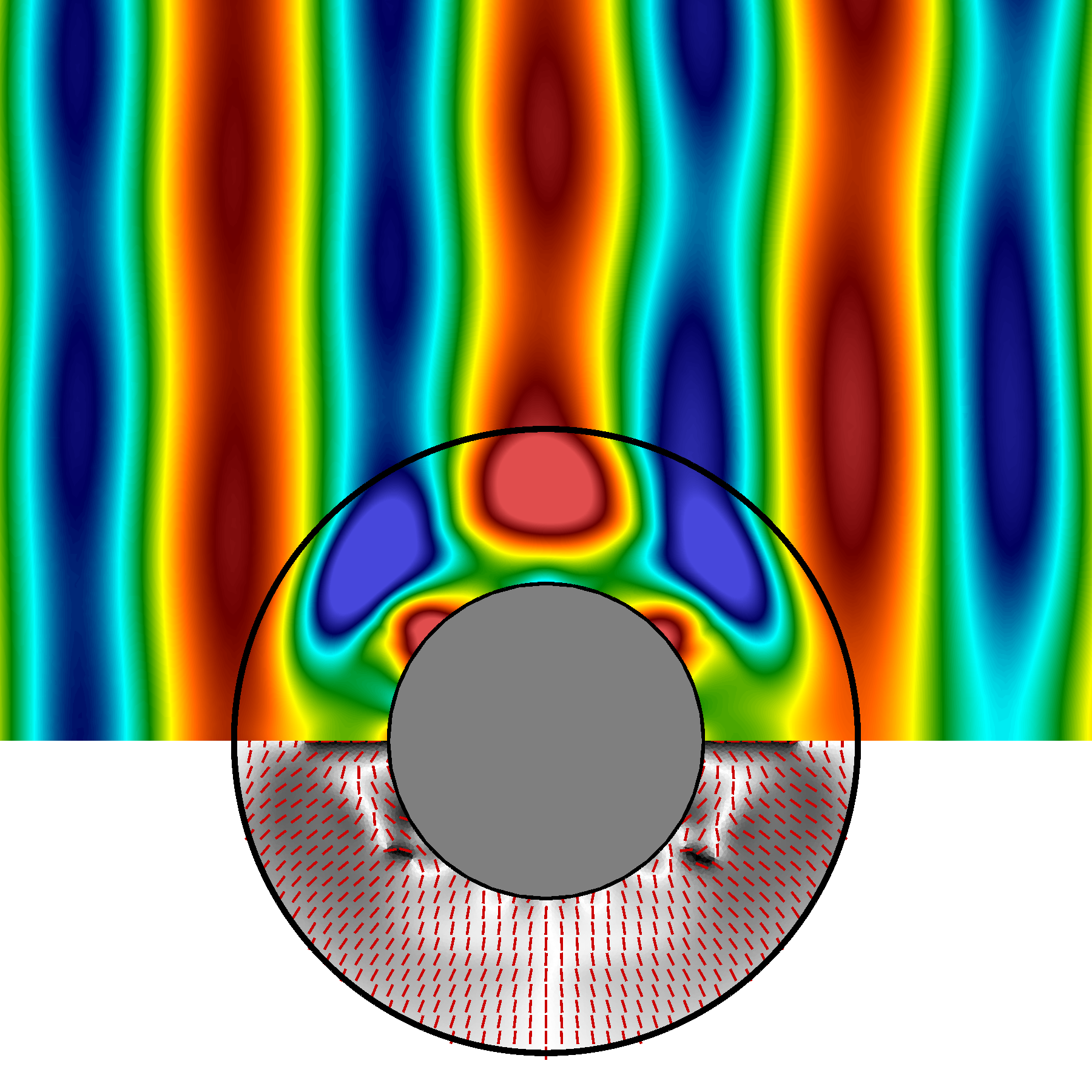}}} \subfigure[close-up \label{subfig:clk:closeup}]{{\includegraphics[width=0.31\textwidth]{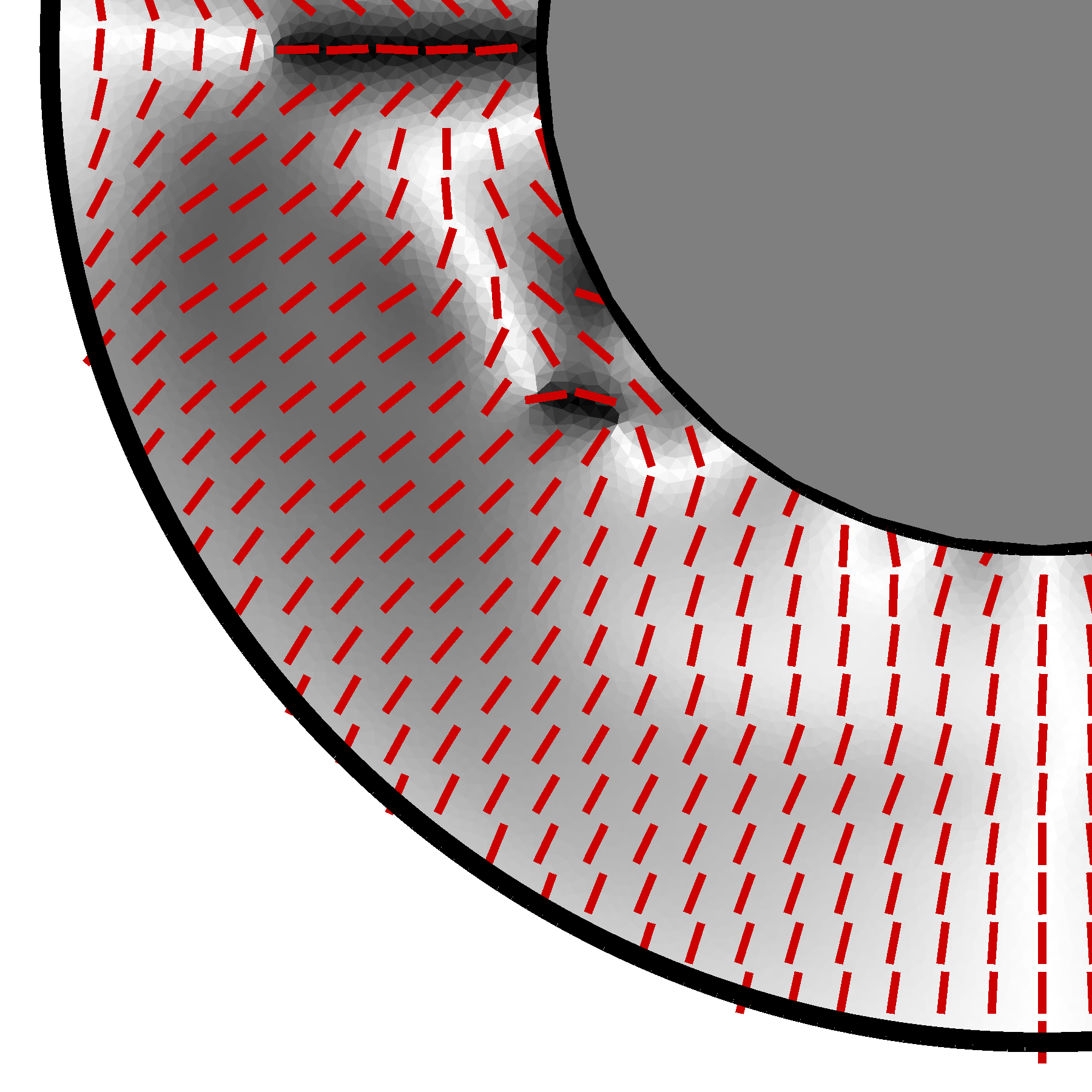}}} \caption{Optimization result for the cloaking of scatterer with anisotropic material} \end{figure} In the upper part of \cref{subfig:clk:inital} and \cref{subfig:clk:final} the total magnetic field is illustrated for the initial and the optimal design configuration, respectively. The backscattering and absorption of the particle including the coating layer is clearly visible for the initial configuration. In the lower \corr{section of both} figures the local orientation of the anisotropic tensor inside the coating layer is visualized. The gray scale colors represent the absolute value of the orientation angle between $0$ and $\frac\pi2$ and the dashes in the closeup (see \cref{subfig:clk:closeup}) support the illustration of the orientation angle. The extinction cross section associated with the optimized anisotropic coating layer is decreased by over $88\percent$ relative to the initial design. \par \corr{As previously mentioned,} we are also interested in \corr{the performance of \cref{alg:SGP}}, if we restrict the design to $L\in\N$ uniformly distributed admissible orientations, \ie $\tfrac{\pi l}{L}$ for all $0\le l\le L-1$. In order to investigate this, the underlying graph of $\GraphSet$ is divided into $L$ edges, \ie $\noEdges=L$, and we obtain the modified admissible set $\AdmSet$ with \begin{align*} \GraphSet = \bigcup_{l=0}^{L-1} \left\{ R\left(\tfrac\pi{L}(\delta+l)\right) B^{(r)} R\left(\tfrac\pi{L}(\delta+l)\right)^T\mid \delta\in[0,1]\right\} \end{align*} and $\NodeImage = \{ R\left(\tfrac{\pi l}{L}\right) B^{(r)} R\left(\tfrac{\pi l}{L}\right)^T\mid 0\le l\le L-1\}$. \corr{To ensure that only points located at nodes are considered}, in theory we could choose a penalty parameter $\gamma>\gamma_{\text{max}}$. \corr{However, as} we know from \cref{rem:discreteSol} that the global minimizer of each sub-problem is located in the set $\NodeImage^\noDesign$ in this case, rather than solving the sub-problem as \corr{described} in \cref{sec:polyparam}, we \corr{can evaluate} the model objective in all nodes and choose the one with the lowest function value for each element. \begin{table}\centering \caption{Progression of the relative cloaking with respect to number of admissible angles}\label{tbl:progression} \begin{tabular}{cc} \toprule number of angles & rel. cloaking \\\midrule 4 & 85.69\percent\\ 12 & 41.35\percent\\ 18 & 36.71\percent\\ 60 & 27.05\percent\\ 180 & 19.80\percent\\ 360 & 12.73\percent\\ continuous & 11.86\percent\\ \bottomrule \end{tabular} \end{table} The relative cloaking after optimization with different numbers of admissible angles is listed in \cref{tbl:progression}. It can be observed that \corr{with an increasing number of orientations, the optimal value of the cost function approaches the optimal value of the continuous problem.} This reveals that despite the \corr{apparent 'brute force'} approach to the solution of the sub-problem, the algorithm is obviously not trapped in local optima introduced by the highly non-convex grayness terms. \begin{figure}\centering \subfigure[18 angles]{ {\includegraphics[,width=0.25\textwidth]{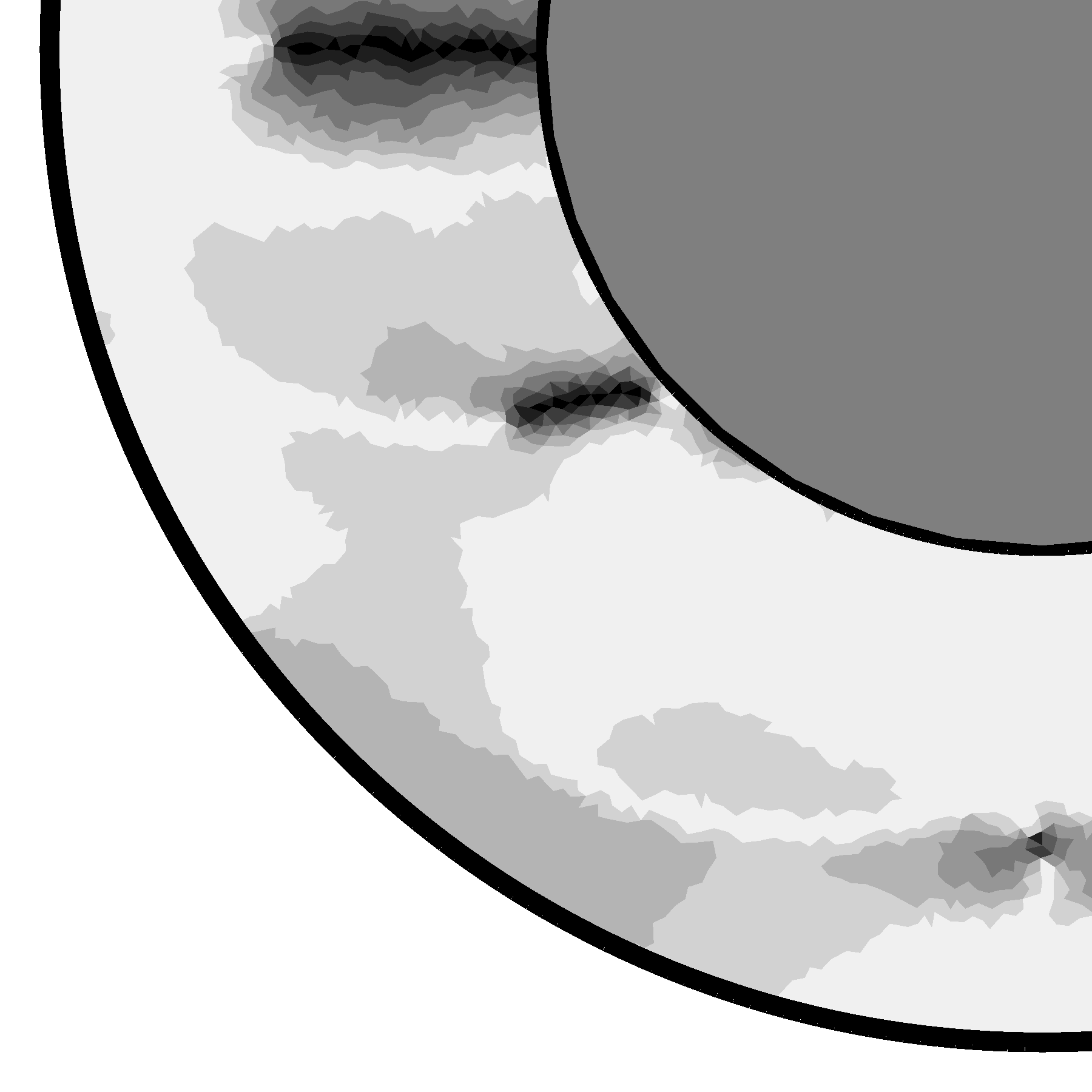}}} \subfigure[180 angles]{ {\includegraphics[,width=0.25\textwidth]{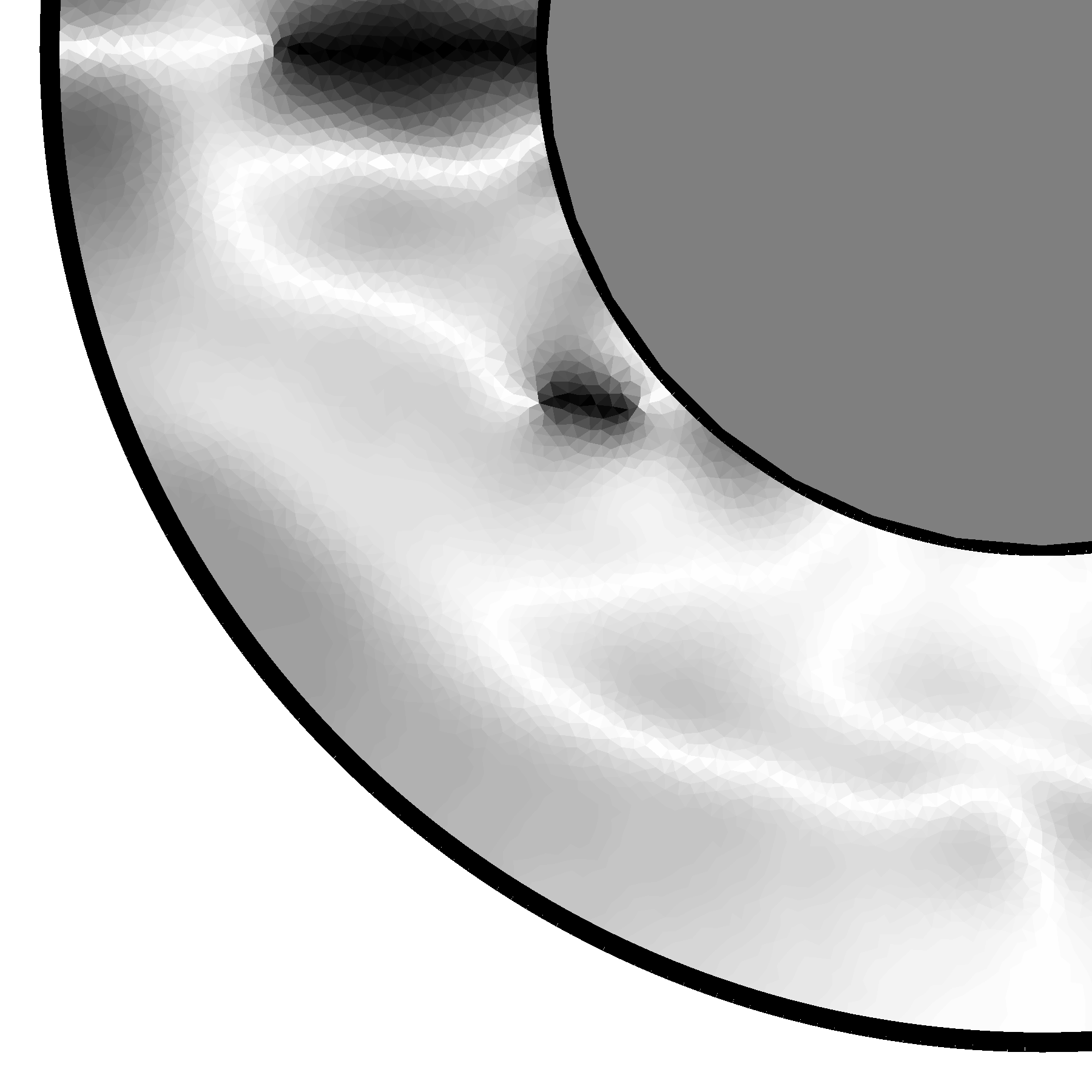}}} \subfigure[continuous]{ {\includegraphics[,width=0.25\textwidth]{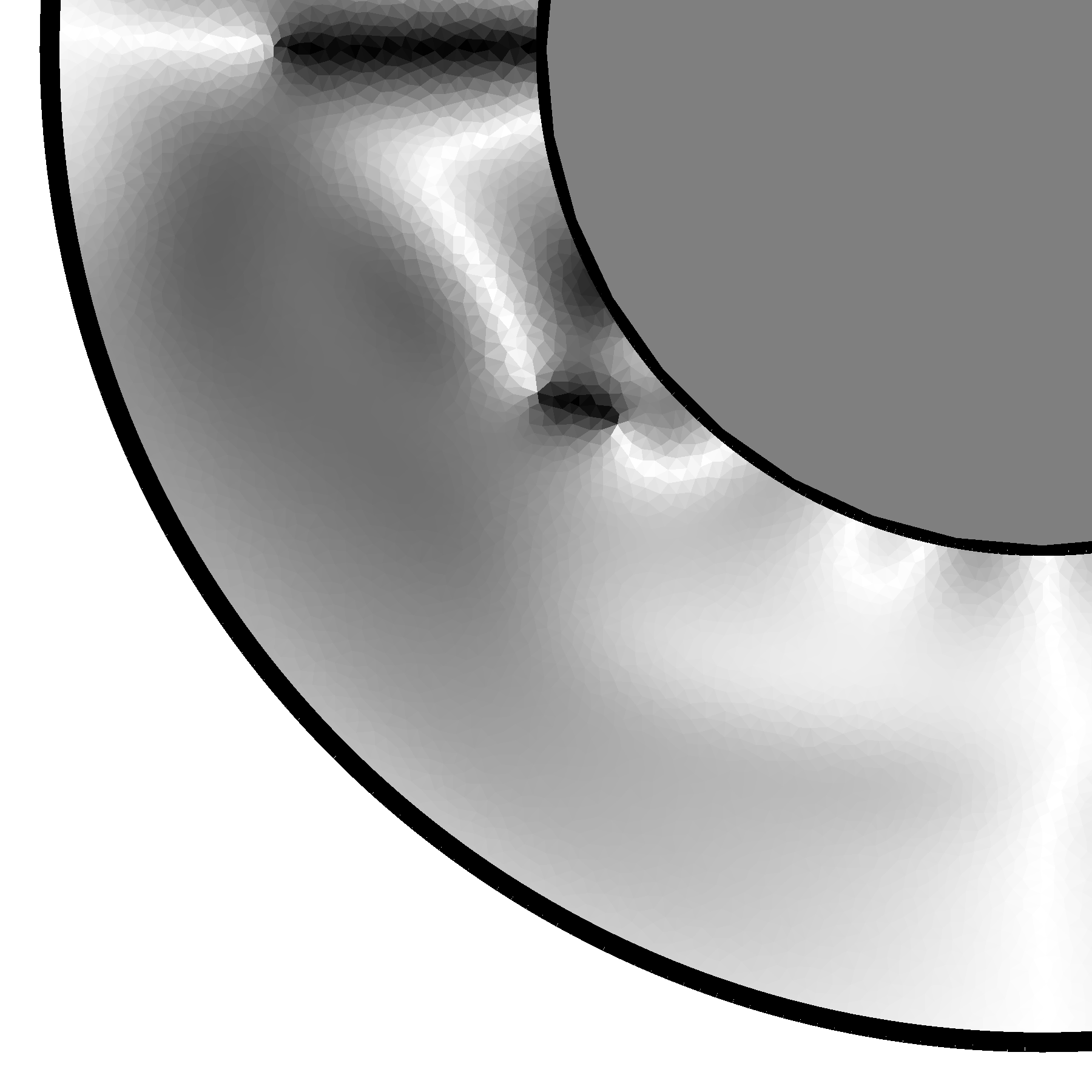}}} \caption{Comparison of optimization results for different number of admissible orientations}\label{fig:contvsdisc} \end{figure} In \cref{fig:contvsdisc} the close up of the optimization results for \corr{18 admissible angle, 180 admissible angles} and the continuous setting are compared.  \subsection{Tomographic reconstruction} In this example we \corr{attempt to} reconstruct a known material distribution by using the electromagnetic field response on the artificial observation boundary $\d U$. \corr{To do this,} we define the physical objective functional following \cite{Cheney1999,Vauhkonen1998} as \begin{equation*} \Jphys(u;u_D) = \int_{\d U} |u(x) - u_D(x)|^2 \dx. \end{equation*} The objective functional measures the distance of a magnetic field $u:\Omega\to\C$ associated with a material configuration $B$ to the magnetic field $u_D:\Omega\to\C$ associated with a reference configuration $B_D$ on the observation boundary $\d U$. After discretization, we obtain \begin{align*} \Jphysdisc(\u; u_D) = \Re\left(\u^H Q \u + 2 \u^H V(u_D) + W(u_D)\right) \end{align*} where $Q\in \R^{\noDoF\times\noDoF},V(u_d)\in \C^{\noDoF}$ and $W(u_D)\in\R$ are defined element-wise for all $1\le i,j \le \noDoF$ by \begin{align*} (Q)_{ij} & := \int_{\d U} \phi_i \phi_j \rmd x, & (V(u_D))_{i} & := \int_{\d U} \phi_i u_D \rmd x, & W(u_D) & := \int_{\d U} |u_D|^2 \rmd x. \end{align*} This functional must be minimized for multiple incident plane waves $u_I(x;\omega,d) = \exp(\imath \omega d \cdot x)$ depending on wave number $\omega$ and direction $d$. Hence, let $\noWaveL$ wave numbers $\omega_1,\dotsc,\omega_\noWaveL$ and $\noDir$ directions $ d_1,\dotsc, d_\noDir$ be given, then the physical objective functional for multiple incident waves $\boldsymbol{J}^p_h$ is \begin{equation*} \boldsymbol{J}^p_h(\u^{11},\dotsc,\u^{\noWaveL\noDir}) = \sum_{l=1}^\noWaveL \sum_{m=1}^\noDir \Jphysdisc(\u^{lm};u^{lm}_D) \end{equation*} where $\u^{lm}$ is the solution of the state equation for wave number $\omega_l$, incident direction $d_m$ with incident wave $u_I(x) = u_I(x; \omega_l,d_m)$ and $u^{lm}_D$ is the corresponding reference solution, \ie $u_D^{lm}(x) = u_D(x; \omega_l,d_m)$.  \subsubsection{Optimization problem} We \corr{again gather} the results of the previous sections, and extend the optimization problem \eqref{eqn:opt:probTMdis} to multiple incident waves: \begin{align*} \left \{ \qquad \begin{aligned} \min_{\B} ~ & \boldsymbol{J}^p_h(\u^{11},\dotsc,\u^{\noWaveL\noDir}) + \eta \Jregdisc(\B) + \gamma \Jgraydisc(\B) & \\ \st ~ &(\SYS(\B) + \SYS_C) \u^{lm} = (\RHS(\B) + \RHS_C) \\ & \text{with } \omega = \omega_l \text{ and } u_I(x) = u_I(x;w_l,d_m)\\ & \corr{\text{where }} \B \text{ is parametrized using polynomial interpolations} \end{aligned}\right. \end{align*} With $\b p^{lm}\in \C^\noDoF$ solving the adjoint equation \begin{align*} (\SYS(\B) + \SYS_C ) \b p ^{lm} & = - 2(Q \u ^{lm} + L(u^{lm}_D))^* \end{align*} for $\omega= \omega_l$, we \corr{are once more able} to give a formula for the derivative of the physical objective in direction $\Y\in\SCK$ for all $k\in\indexDesign$ as follows: \begin{align*} \tot{\boldsymbol{J}^p_h(\u^{11},\dotsc,\u^{\noWaveL\noDir})}{(\B)_k} [(\Y)_k] & = \sum_{l=1}^\noWaveL \sum_{m=1}^\noDir \Re\left ( \int_{\TriI{k}}(\nabla_{\!\!h} \u ^{lm} + \nabla u_I^{lm} )^T (\Y)_k\nabla_{\!\!h} \b p^{lm} \rmd x \right ). \end{align*} As in the previous example, we have used $\nabla_{\!\!h} \b u = \sum_{i=1}^\noDoF u_i\nabla \phi_i$, $\nabla_{\!\!h} \b p = \sum_{i=1}^\noDoF p_i\nabla \phi_i$ and additionally $u_I^{lm}(x) = u_I(x;\omega_l,d_m)$. \corr{With this as a basis}, again sub-problems can be formulated and solved using the techniques described in \cref{sec:polyparam}.  \subsection{Numerical Results} The design domain $\Omega_D$ is a \corr{ball} with radius $0.4$ and is contained in a square with side length $2$. Furthermore a perfectly matched layer of thickness $1$ is used and the PML function $s(t)$ is defined as in the previous example. The scattered electromagnetic field $u_D$ of the reference configuration $B_D$ is calculated on the same mesh, but perturbed \corr{at} every spatial point with a weighted normal probability density function with mean value $0$ and variance $1$. The reference magnetic field $u_D$ is used for the evaluation of the objective function on the observation boundary defined by a sphere with radius $0.8$. Furthermore\corr{,} the scatterer is illuminated from $8$ directions in \SI{45}{\degree} steps and $8$ uniformly distributed wavelengths in the interval $[0.4,0.7]$. \par The set of admissible material tensors is the cyclic graph with three edges ($\indexEdges =\{1,\dotsc,3\}$) and linear interpolation of the isotropic materials $B^{(1)} = \mathds 1$, $B^{(2)} = (2)^{-2}\mathds1$ and $B^{(3)} = (1 +2\imath)^{-2}\mathds1$. \par \begin{figure}\centering \hfill \subfigure[material distribution $B_D$\label{subfig:mattarget}]{ {{\centering \parbox[b][0.25\textwidth][c]{0.5cm}{\includegraphics[height=0.16\textwidth]{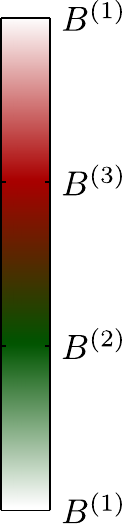}} \includegraphics[height=0.25\textwidth,angle=0]{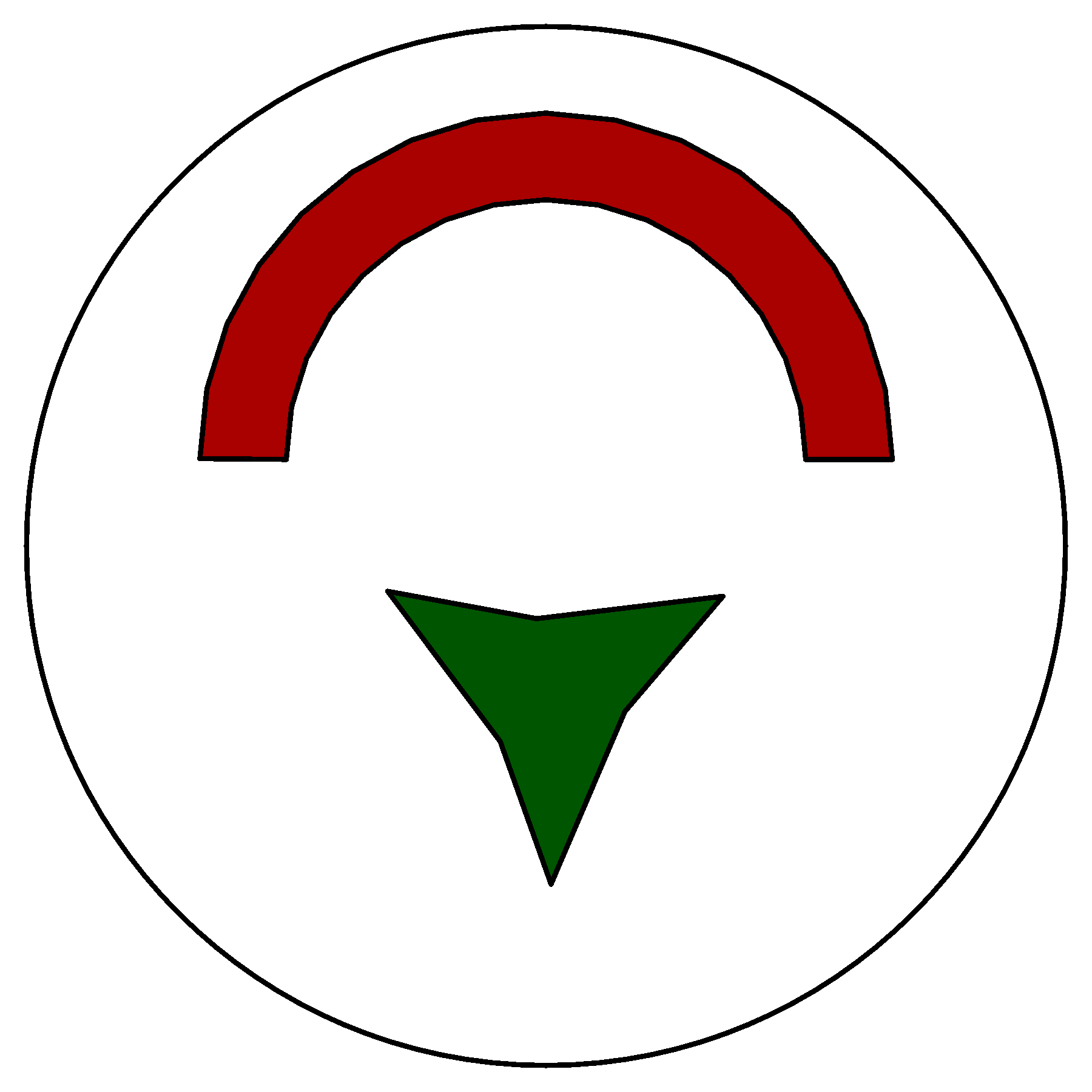}}}} \hfill \subfigure[noisy scattered target field $u_D$\label{subfig:magnettarget}]{ \hspace{1em}{{\centering \parbox[b][0.25\textwidth][c]{0.5cm}{\includegraphics[height=0.16\textwidth]{programresults/cbar_rainbow_-2x2.pdf}} \includegraphics[height=0.25\textwidth,angle=0]{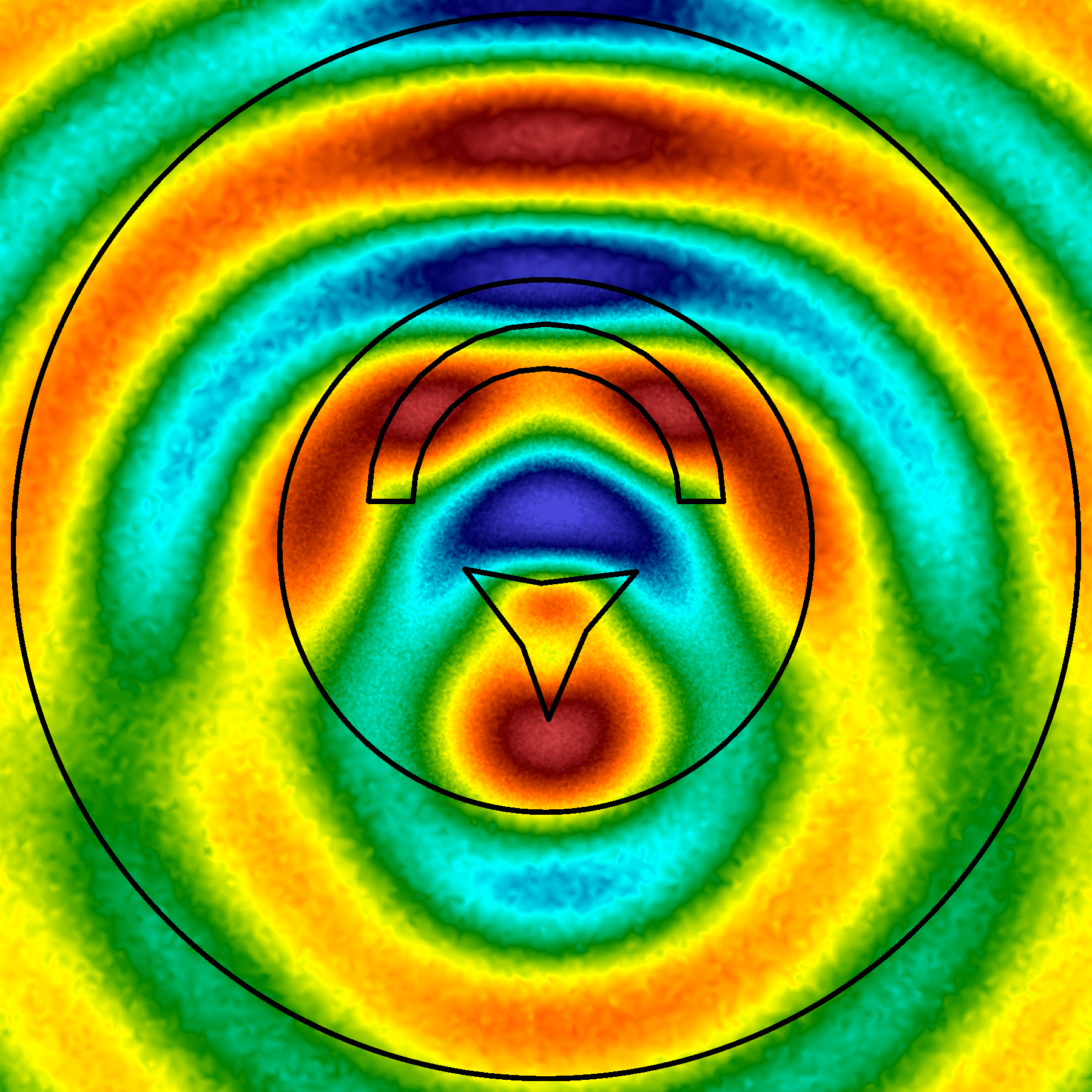}}}\hspace{1em}} \hfill \caption[asd]{Reference material configuration and reference magnetic field including noise.} \end{figure} \par \cref{subfig:mattarget} illustrates the material distribution $B_D$ inside the circular design domain, where white corresponds to $B^{(1)}$, green corresponds to $B^{(2)}$ and red corresponds to $B^{(3)}$, respectively. Moreover the outline of the design domain and material distribution is marked by black lines. The noisy scattered magnetic field for a fixed wavelength and illumination direction is depicted in \cref{subfig:magnettarget}. Here, the \corr{outermost} circle is the observation boundary $\d U$ where the objective functional is evaluated. \par \begin{figure}\centering \hfill \parbox[b][0.25\textwidth][c]{0.5cm}{\includegraphics[height=0.16\textwidth]{programresults/cbar_mat_name.pdf}} \hfill \subfigure[$\eta =0.1$]{\includegraphics[height=0.25\textwidth,angle=0]{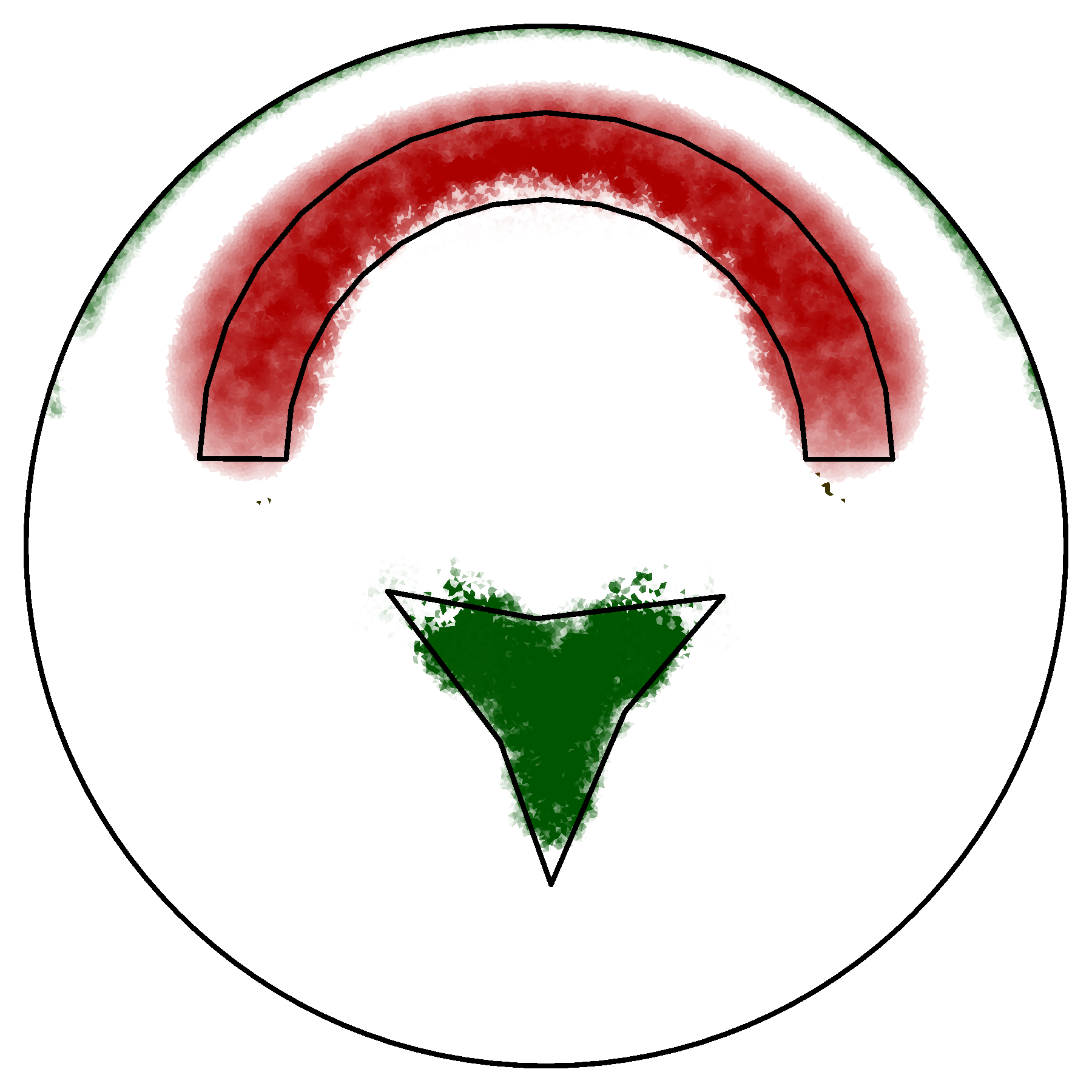}} \hfill \subfigure[$\eta =1$]{\includegraphics[height=0.25\textwidth,angle=0]{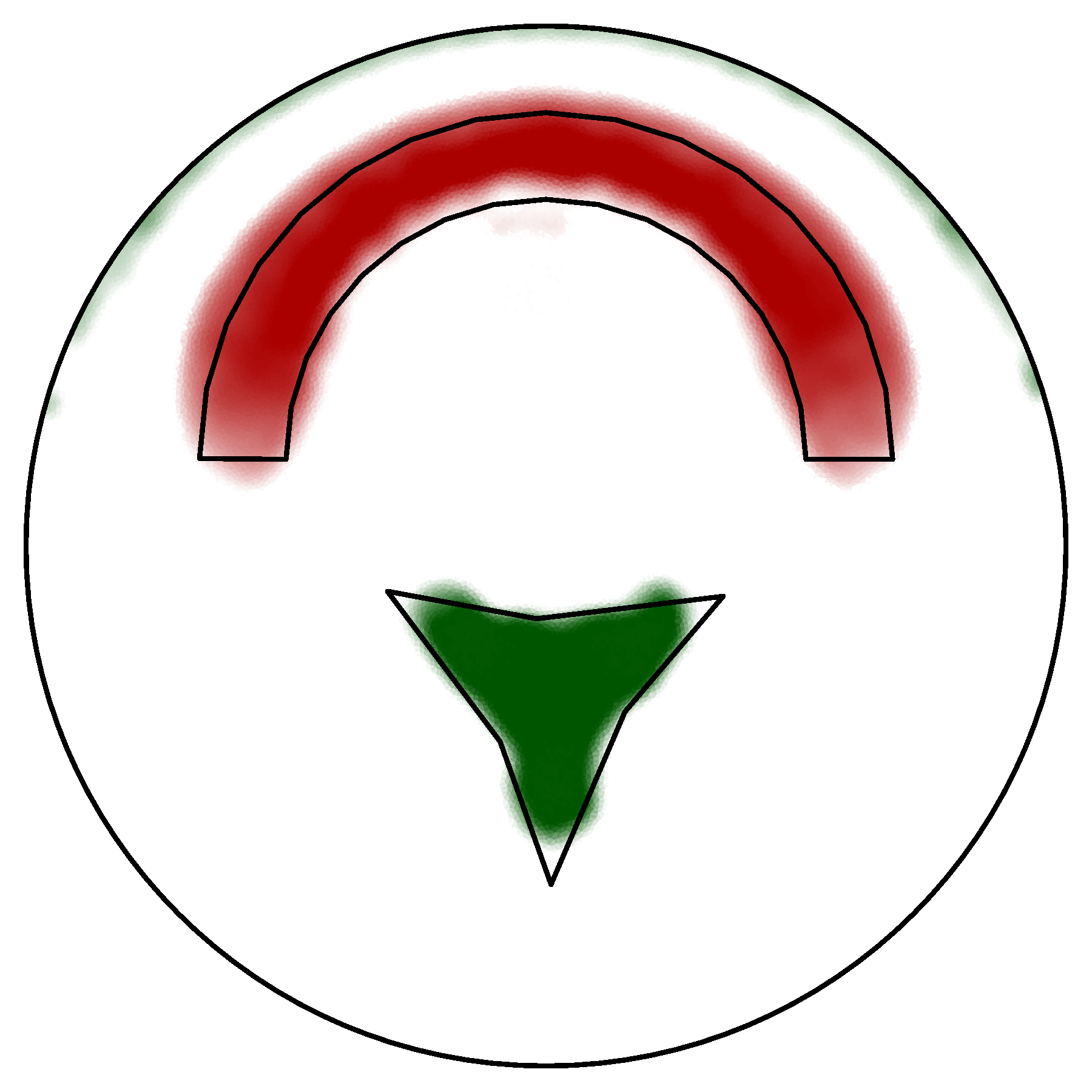}} \hfill \subfigure[$\eta =10$]{\includegraphics[height=0.25\textwidth,angle=0]{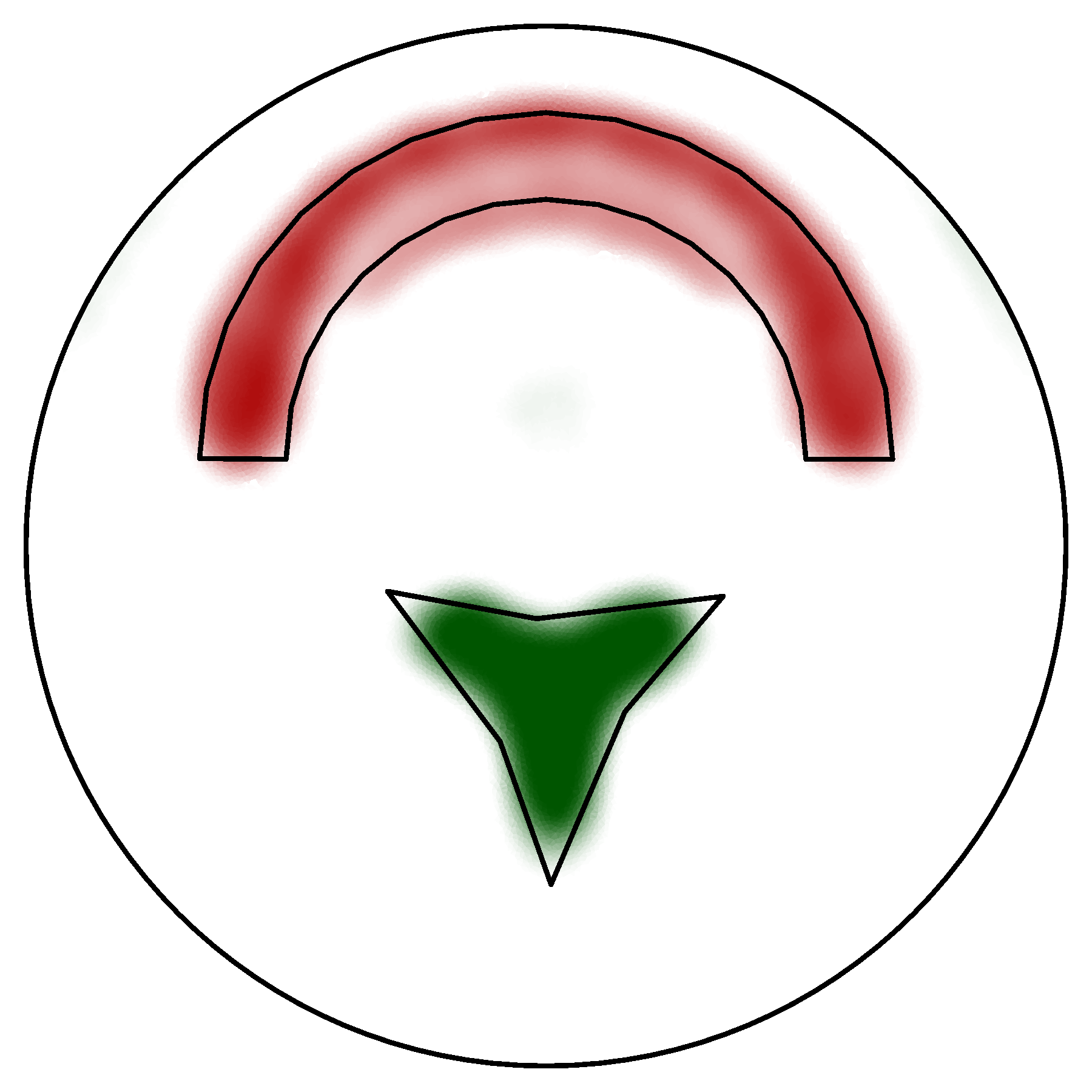}} \hfill \caption[asd]{Effect of filter penalization factor $\eta$ on optimization result after 100 iterations, $\gamma= \num{e-6}$, $r_0 =0.01$, colored by material index\label{fig:regcomp}} \end{figure} \par First we wish to analyze the influence of the regularization. We set the circular filter radius to $r_0=0.01$, choose a grayness parameter of $\gamma = \num{e-6}$ and perform the optimization with various choices of regularization constants $\eta \in \{\num{0.1},\num{1},\num{10}\}$. In \cref{fig:regcomp}, the optimization result after 100 iterations is depicted for all choices; the effect of the regularization becomes apparent at the blurred interfaces. \corr{While the influence of the filter radius was also studied, this is not described here.} \par \begin{figure}\centering \hfill \subfigure[objective functional\label{subfig:objective}]{{\includegraphics[height=0.25\textwidth,angle=0]{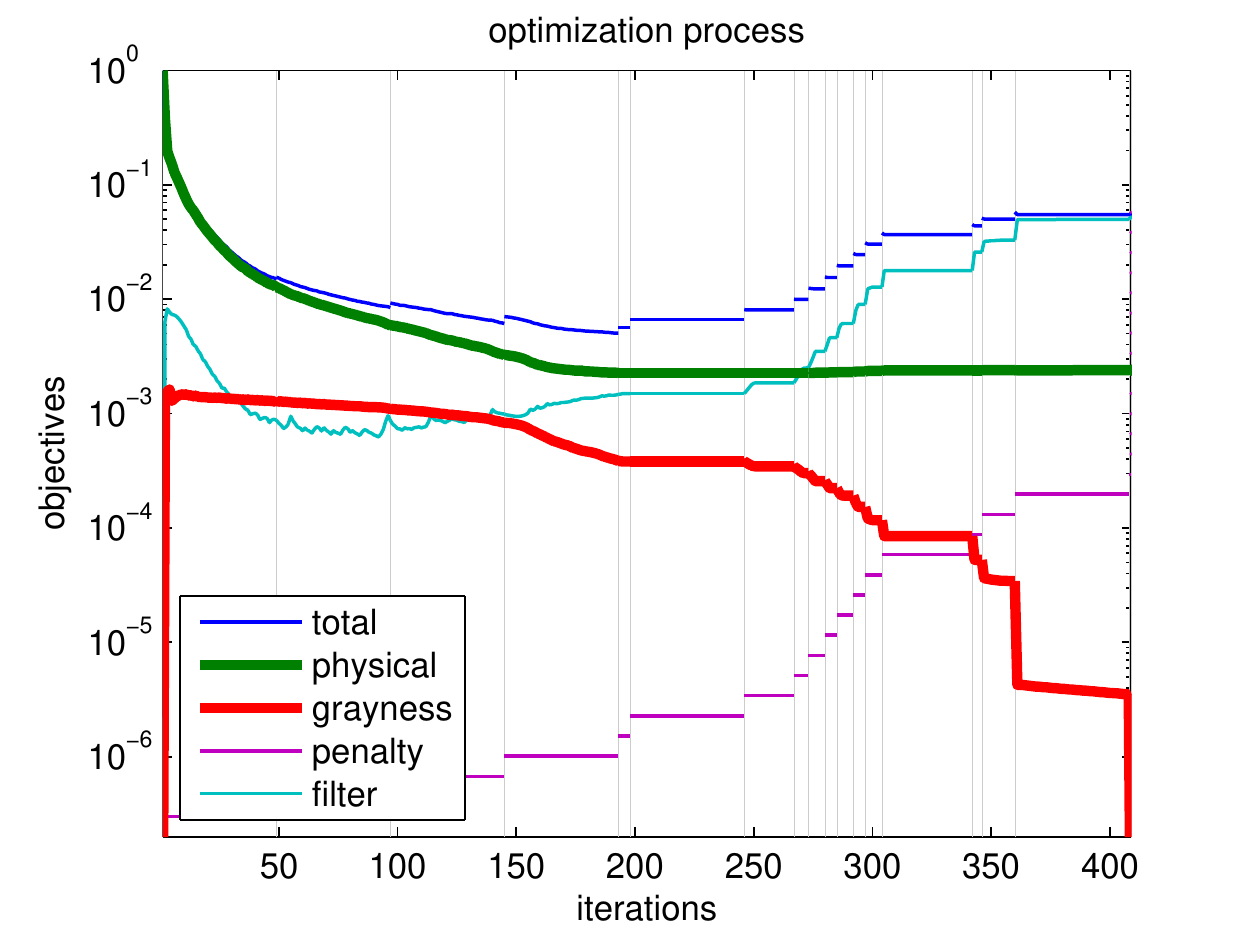}}} \hfill \subfigure[optimization result\label{subfig:optresult}]{ \includegraphics[height=0.25\textwidth,angle=0]{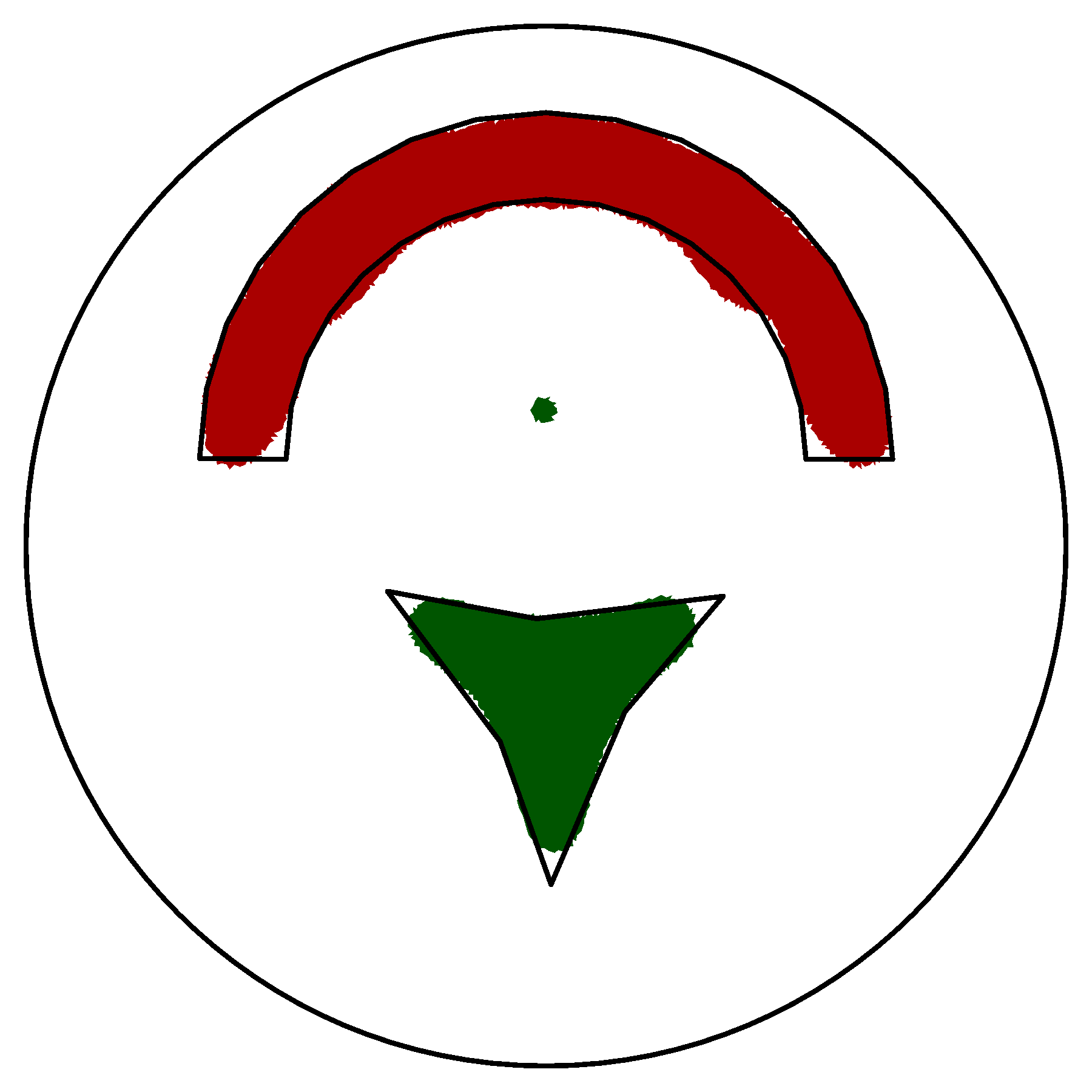}} \hfill \subfigure[magnetic field on $\d U$ \label{subfig:observ}]{{\includegraphics[height=0.25\textwidth,angle=0]{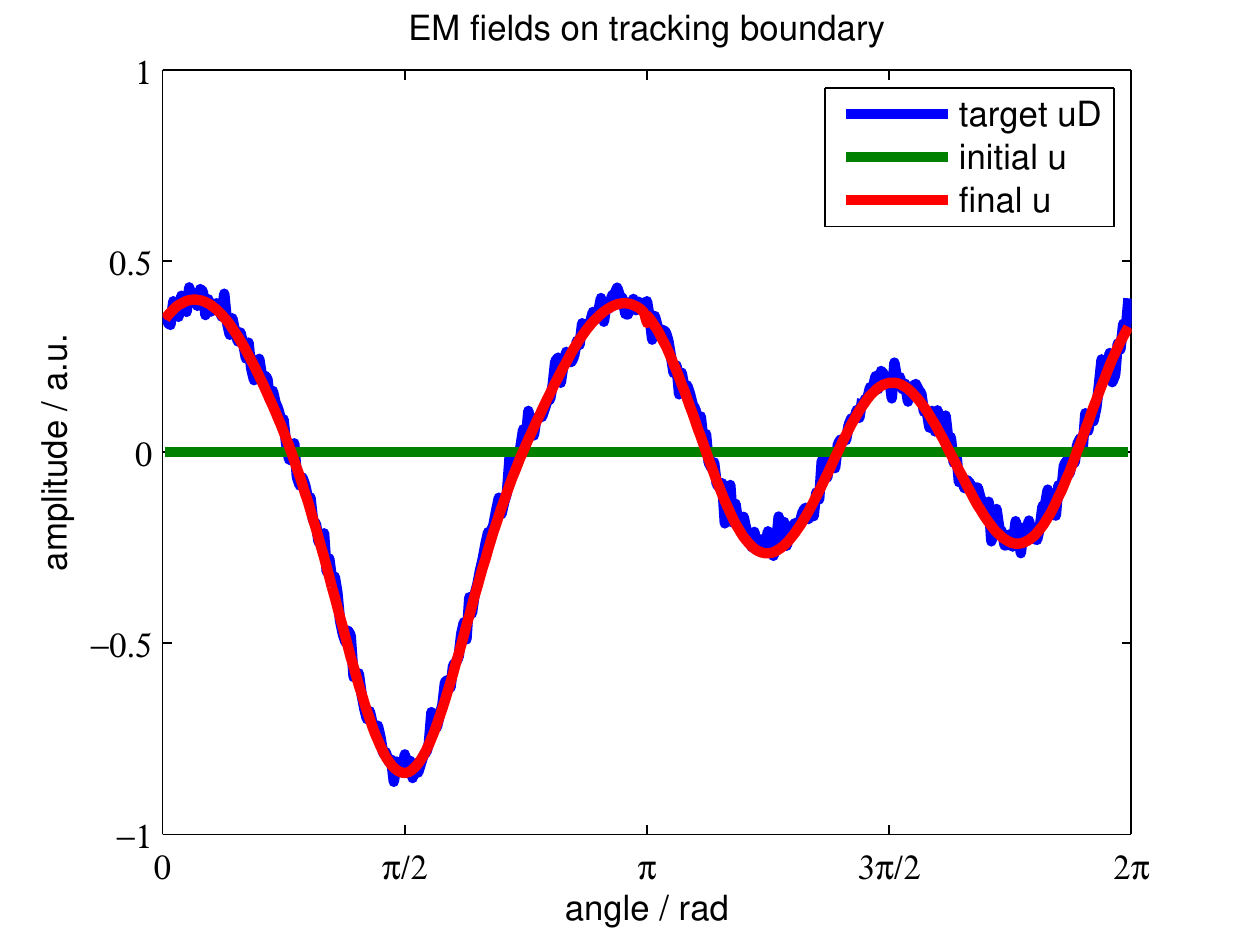}}} \hfill \caption{Evolution of the objective functional with continuation procedure and magnetic field quantities on the observation boundary $\d U$ for single wavelength and illumination direction} \end{figure} \par During the final optimization process\corr{,} a continuation procedure like \corr{that} in \cite{Sigmund1998} was executed, \ie the algorithm \corr{was} repeatedly restarted with an increased grayness penalty factor $\gamma$ and the previous optimization result as starting point. Based on the previous investigations we \corr{fixed} $\eta=1$ and $r_0=0.01$, and \corr{initialized} the material tensor in the design domain with background material. In \cref{subfig:objective} the evolution of the objective functional value, normalized to the inital state, is depicted. The objective functional ({\tt total}) is composed of the extinctional cross section $\b J^p_h$ ({\tt physical}), the regularization term $\Jregdisc$ ({\tt filter}) and the grayness term $\Jgraydisc$ ({\tt grayness}); all contributions are plotted seperately. Due to the increment of the grayness penalty factor $\gamma$ ({\tt penalty}), the total objective functional increases but the physical and grayness contributions \corr{decrease successively}. Additionally\corr{,} the vertical lines in \cref{subfig:objective} mark the continuation procedure where either a predefined maximum number of iterations is exceeded or a stopping criterion is reached. The continuation process is ended when a so-called black/white solution ($\Jgraydisc=0$) is achieved. Thus\corr{,} we finally obtain a material distribution with tensor values corresponding to the nodes of the underlying graph (\cref{subfig:optresult}). The difference between \corr{the} target and optimized magnetic response, \ie the physical objective functional, drops to $0.2\percent$. This can be also qualitively observed in \cref{subfig:observ}, where for a single wavelength and illumination direction the magnetic response $u$ and the noisy target field $u_D$ are compared along the circular observation boundary $\d U$. \par Finally, we study the effect of the filter penalization on the optimization result in combination with the continuation strategy. \corr{ \cref{fig:filterpenopt} depicts the results of the optimization with continuation process, where the filter penalty factor $\eta\in \{0.1,10,100\}$ is varied.} \corr{It is evident} that the overall shape of the scatterer \corr{has been} more or less reconstructed. For the smallest factor $\eta = 0.1$, we see a ragged boundary, whereas for the largest factor $\eta=100$ the filtering was too strong and \corr{dominated} the tracking objective. \par In all cases the sharp features, like the edges of the star, are poorly reconstructed, which may be caused by the low number of wavelengths and illumination directions. \par \begin{figure}\centering \hfill \parbox[b][0.25\textwidth][c]{0.5cm}{\includegraphics[height=0.16\textwidth]{programresults/cbar_mat_name.pdf}} \hfill {\subfigure[$\eta=0.1$]{\includegraphics[height=0.25\textwidth,angle=0]{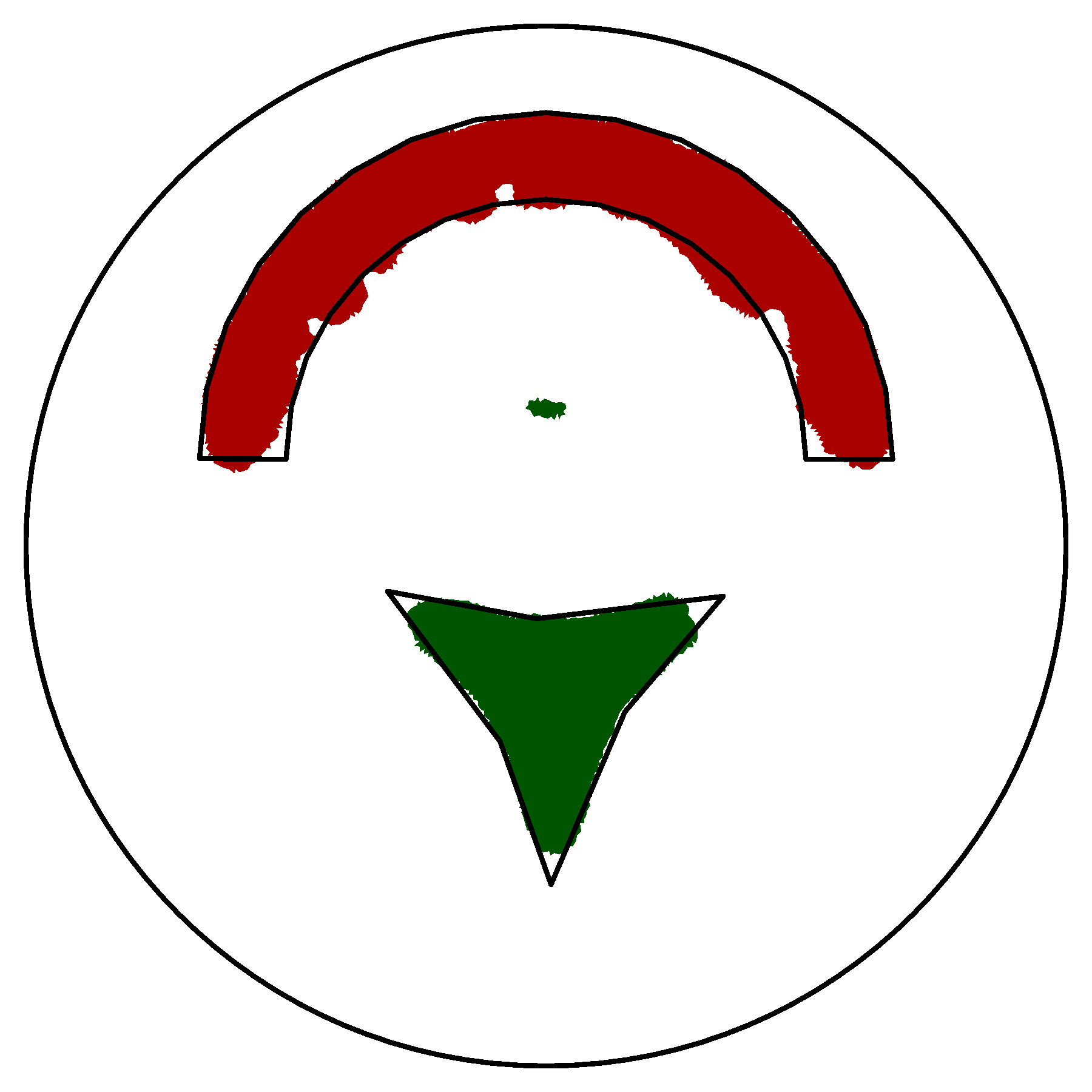}}} \hfill {\subfigure[$\eta=10$]{\includegraphics[height=0.25\textwidth,angle=0]{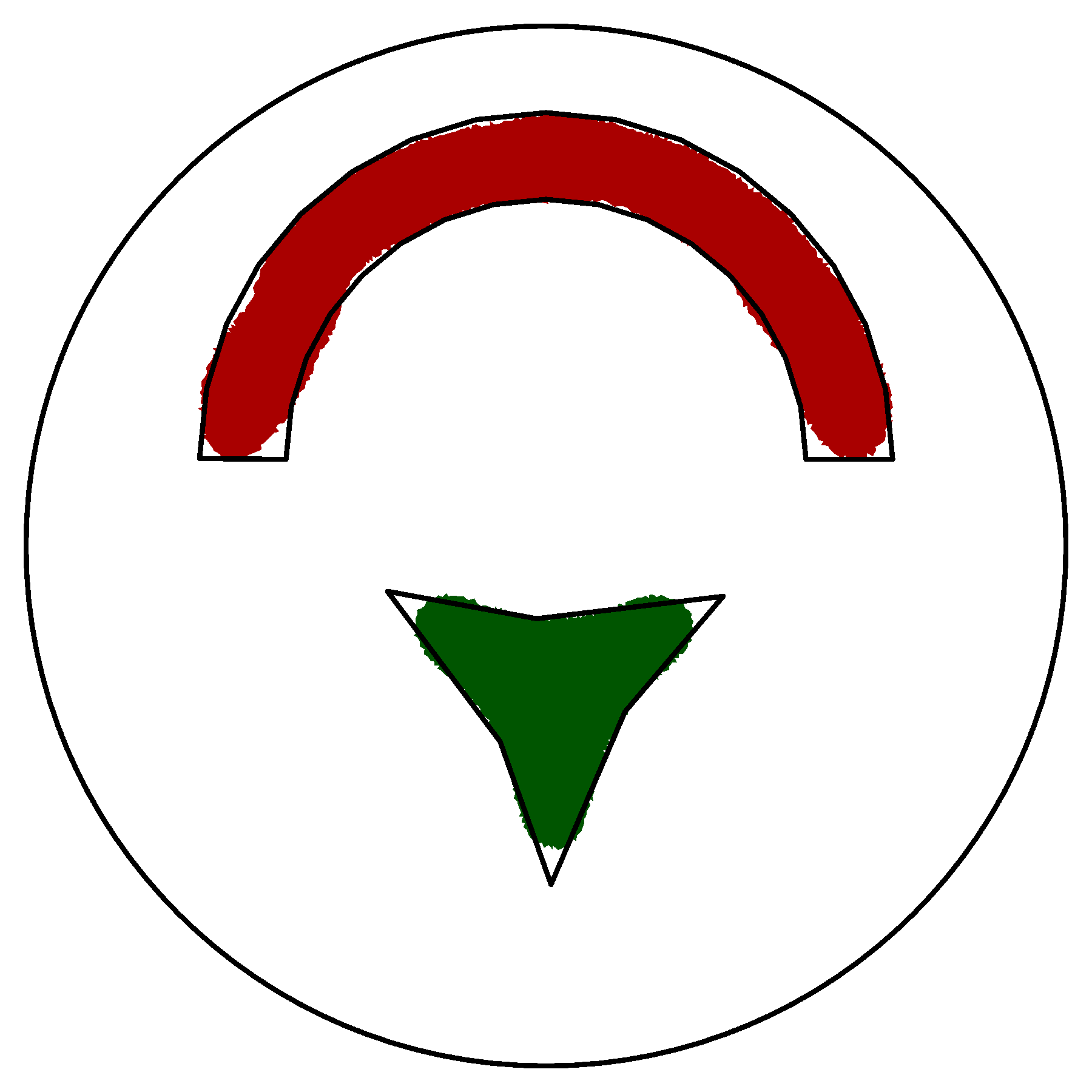}}} \hfill {\subfigure[$\eta=100$]{\includegraphics[height=0.25\textwidth,angle=0]{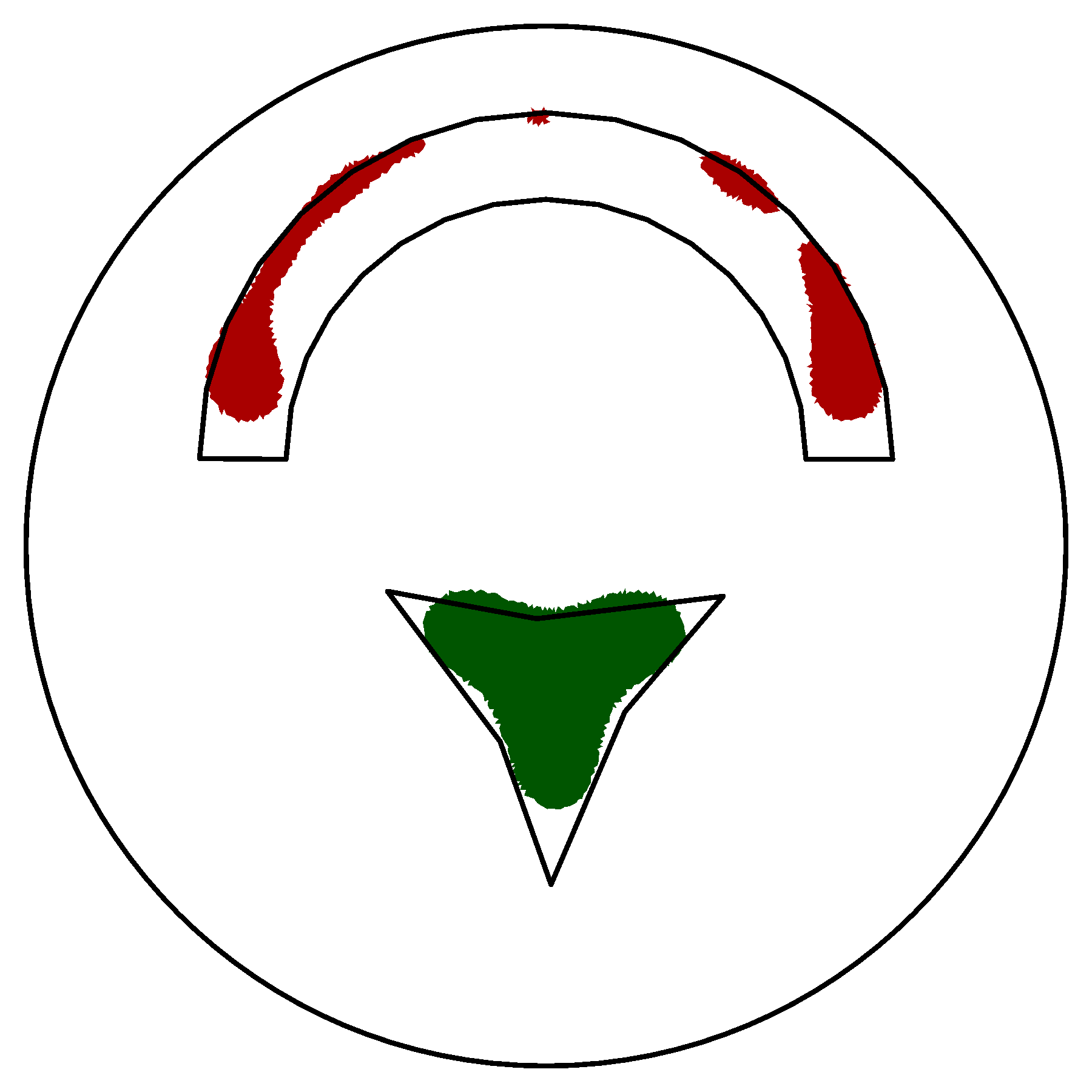}}} \hfill \caption{Optimization result with variation of the filter penalty $\eta$ \label{fig:filterpenopt}} \end{figure}  \section{Concluding Remarks} \corr{We have proposed a new algorithm for the solution of material optimization problems in electromagnetics.} The algorithm is flexible in the sense that it can be applied to material problems of \corr{a} discrete and continuous nature. Theoretical properties of the algorithm \corr{such as} global convergence have been discussed \corr{and} it was further shown in the course of numerical experiments that poor local optima introduced by seriously non-linear parametrizations can be avoided. \corr{An open question remains: how can the improvement of the new algorithm over the application of general purpose optimization algorithms be quantified? This will be investigated in the near future, based on extended numerical experiments and a thorough mathematical investigation of the properties of limit points.} \corr{In terms of} the application side it \corr{would be a natural next step} to apply the concept to three dimensional problems and to \corr{consider} efficient parallelization. The algorithmic concept discussed in the article seems particularly suited for the latter due to the block separable structure of the sub-problem, which \corr{must} be solved in each major iteration. \let\c\oldc\hbadness 20000\relax\bibliography{MOTESA_abbr}\bibliographystyle{siamplain} 
\begin{thebibliography}{10}

\bibitem{Bendsoe1994}
{\sc M.~P. Bends{\o}e, J.~M. Guedes, R.~B. Haber, P.~Pedersen, and J.~E.
  Taylor}, {\em {An Analytical Model to Predict Optimal Material Properties in
  the Context of Optimal Structural Design}}, J. Appl. Mech., 61 (1994),
  p.~930.

\bibitem{Berenger1994}
{\sc J.~P. Berenger}, {\em {A perfectly matched layer for the absorption of
  electromagnetic waves}}, J. Comput. Phys., 114 (1994), pp.~185--200.

\bibitem{Byun2004}
{\sc J.-K. Byun, J.-H. Lee, and I.-H. Park}, {\em {Node-Based Distribution of
  Material Properties for Topology Optimization of Electromagnetic Devices}},
  IEEE Trans. Magn., 40 (2004), pp.~1212--1215.

\bibitem{Cheney1999}
{\sc M.~Cheney, D.~Isaacson, and J.~C. Newell}, {\em {Electrical Impedance
  Tomography}}, SIAM Rev., 41 (1999), pp.~85--101.

\bibitem{Diaz2010}
{\sc A.~R. Diaz and O.~Sigmund}, {\em {A topology optimization method for
  design of negative permeability metamaterials}}, Struct. Multidiscip. Optim.,
  41 (2010), pp.~163--177.

\bibitem{Edelman1995}
{\sc A.~Edelman and H.~Murakami}, {\em {Polynomial roots from companion matrix
  eigenvalues}}, Math. Comput., 64 (1995), pp.~763--763.

\bibitem{Gill2005a}
{\sc P.~E. Gill, W.~Murray, and M.~A. Saunders}, {\em {SNOPT: An SQP Algorithm
  for Large-Scale Constrained Optimization}}, SIAM J. Optim., 12 (2002),
  pp.~979--1006.

\bibitem{Greifenstein2016}
{\sc J.~Greifenstein and M.~Stingl}, {\em {Simultaneous parametric material and
  topology optimization with constrained material grading}}, Struct.
  Multidiscip. Optim., 54 (2016), pp.~985--998.

\bibitem{Haber1996a}
{\sc R.~B. Haber, C.~S. Jog, and M.~P. Bendsoe}, {\em {A new approach to
  variable-topology shape design using a constraint on perimeter}}, Struct.
  Optim., 11 (1996), pp.~1--12.

\bibitem{Hassan2014}
{\sc E.~Hassan, E.~Wadbro, and M.~Berggren}, {\em {Topology Optimization of
  Metallic Antennas}}, IEEE Trans. Antennas Propag., 62 (2014), pp.~2488--2500.

\bibitem{Hvejsel2011}
{\sc C.~F. Hvejsel and E.~Lund}, {\em {Material interpolation schemes for
  unified topology and multi-material optimization}}, Struct. Multidiscip.
  Optim., 43 (2011), pp.~811--825.

\bibitem{Jackson1963}
{\sc J.~D. Jackson}, {\em {Classical Electrodynamics}}, Wiley, New York,
  3rd~ed., 2004.

\bibitem{OhinKwon2002}
{\sc O.~Kwon, E.~J. Woo, J.-R. Yoon, and J.~K. Seo}, {\em {Magnetic resonance
  electrical impedance tomography (MREIT): simulation study of J-substitution
  algorithm}}, IEEE Trans. Biomed. Eng., 49 (2002), pp.~160--167.

\bibitem{Matlab2014}
{\sc MATLAB}, {\em {v8.3.0.532 (R2014a)}}, The MathWorks Inc., Nantick,
  Massachusetts, United States, 2014.

\bibitem{Mishchenko2014}
{\sc M.~I. Mishchenko}, {\em {Electromagnetic Scattering by Particles and
  Particle Groups: An Introduction}}, Cambridge University Press, Cambridge,
  2014.

\bibitem{Monk2003}
{\sc P.~Monk}, {\em {Finite Element Methods for Maxwell's Equations}},
  Clarendon Press, Oxford; New York, 1st~ed., 2003.

\bibitem{Pedersen1989}
{\sc P.~Pedersen}, {\em {On optimal orientation of orthotropic materials}},
  Struct. Optim., 1 (1989), pp.~101--106.

\bibitem{Pendry}
{\sc J.~B. Pendry}, {\em {Controlling Electromagnetic Fields}}, Science (80-.
  )., 312 (2006), pp.~1780--1782.

\bibitem{Ringertz1993}
{\sc U.~T. Ringertz}, {\em {On finding the optimal distribution of material
  properties}}, Struct. Optim., 5 (1993), pp.~265--267.

\bibitem{Shewchuk1996}
{\sc J.~R. Shewchuk}, {\em {Triangle: Engineering a 2D quality mesh generator
  and Delaunay triangulator}}, in Appl. Comput. Geom. Towar. Geom. Eng.,
  vol.~1148, Springer, Berlin; Heidelberg, 1996, pp.~203--222.

\bibitem{Sigmund1997a}
{\sc O.~Sigmund}, {\em {On the Design of Compliant Mechanisms Using Topology
  Optimization}}, Mech. Struct. Mach., 25 (1997), pp.~493--524.

\bibitem{Sigmund1998}
{\sc O.~Sigmund and J.~Petersson}, {\em {Numerical instabilities in topology
  optimization: A survey on procedures dealing with checkerboards,
  mesh-dependencies and local minima}}, Struct. Optim., 16 (1998), pp.~68--75.

\bibitem{Stegmann2005}
{\sc J.~Stegmann and E.~Lund}, {\em {Discrete material optimization of general
  composite shell structures}}, Int. J. Numer. Methods Eng., 62 (2005),
  pp.~2009--2027.

\bibitem{Stingl2009a}
{\sc M.~Stingl, M.~Ko{\v{c}}vara, and G.~Leugering}, {\em {A Sequential Convex
  Semidefinite Programming Algorithm with an Application to Multiple-Load Free
  Material Optimization}}, SIAM J. Optim., 20 (2009), pp.~130--155.

\bibitem{Svanberg1987}
{\sc K.~Svanberg}, {\em {The method of moving asymptotes—a new method for
  structural optimization}}, Int. J. Numer. Methods Eng., 24 (1987),
  pp.~359--373.

\bibitem{Vauhkonen1998}
{\sc M.~Vauhkonen, D.~Vadasz, P.~Karjalainen, E.~Somersalo, and J.~Kaipio},
  {\em {Tikhonov regularization and prior information in electrical impedance
  tomography}}, IEEE Trans. Med. Imaging, 17 (1998), pp.~285--293.

\bibitem{Zienkiewicz2013}
{\sc O.~Zienkiewicz, R.~Taylor, and J.~Zhu}, {\em {The finite element method:
  its basis and fundamentals}}, Elsevier, Amsterdam; Boston; Heidelberg,
  6th~ed., 2005.

\bibitem{Zowe1997}
{\sc J.~Zowe, M.~Ko{\v{c}}vara, and M.~P. Bends{\o}e}, {\em {Free material
  optimization via mathematical programming}}, Math. Program., 79 (1997),
  pp.~445--466.

\end{thebibliography}
\end{document}